\documentclass[11pt]{amsart}
\usepackage{mathrsfs}
\usepackage{amsmath,amstext,amsfonts,verbatim}
\usepackage{cite}
\usepackage{amssymb}

\usepackage[latin1]{inputenc}
\usepackage{amsmath,amsthm,amssymb}

\newtheorem{theorem}{Theorem}[section]
\newtheorem{lemma}[theorem]{Lemma}
\newtheorem{proposition}[theorem]{Proposition}
\newtheorem{corollary}[theorem]{Corollary}
\theoremstyle{definition}

\theoremstyle{remark}
\newtheorem{remark}[theorem]{Remark}
\numberwithin{equation}{section}

\title[Fock space on $\mathbb{C}^\infty$ and Bose-Fock space]{Fock space on $\mathbb{C}^\infty$ and Bose-Fock space}
\author[B. D. Wick and S. Wu]{Brett D. Wick and Shengkun Wu}

\address{Brett D. Wick, Department of Mathematics and Statistics, Washington University in St. Louis, MO 63130, USA}
\email{ wick@math.wustl.edu}
\address{Shengkun Wu, College of Mathematics and Statistics, Chongqing University, Chongqing, 401331, PR China}
\email{shengkunwu@foxmail.com}

 \thanks{Brett D. Wick's research is supported in part by a National Science Foundation DMS grants \# 1560955 and \# 1800057 and Australian Research Council -- DP 190100970. Shengkun Wu's research is supported by CSC201906050022.}
 \thanks{}
 \thanks{}

\date{}

\begin{document}

\begin{abstract}In this paper, we introduce the Fock space on $\mathbb{C}^{\infty}$ and obtain an isomorphism between the Fock space on
$\mathbb{C}^{\infty}$ and Bose-Fock space. Based on this isomorphism, we obtain representations of some operators on the Bose-Fock space and answer a question in \cite{coburn1985}.
As a physical application, we study the Gibbs state.
\end{abstract} \maketitle

\section{Introduction}
In \cite{Bargmann}, Bargman introduced the Fock space on $\mathbb{C}^n$ and discussed its connection with quantum mechanics.
In \cite{coburn1985}, Berger and Coburn studied the operators on the Fock space on $\mathbb{C}^n$. In the last section of that paper, the authors asked a question: can the analysis in this paper be applied in the physically interesting case where $\mathbb{C}^{n}$ is replaced by an infinite-dimensional Hilbert space? However, in this paper, we will use $\mathbb{C}^{\infty}$ to replace $\mathbb{C}^{n}$ instead of infinite-dimensional Hilbert space. By this replacement, we will show that the the Fock space on $\mathbb{C}^{\infty}$ is isomorphic to the Bose-Fock space. Then, we are going to generalize some conclusions in \cite{coburn1985} and give a physical application.

The Bose-Fock space is used to describe the states of bosons in quantum mechanics, for the details of this space we refer to \cite{Bartteli}. If $\mathcal{H}$ is a separable Hilbert space, the Full Fock space over $\mathcal{H}$ is the complete tensor algebra over $\mathcal{H}$:
$$
\mathcal{F}(\mathcal{H})=\bigoplus_{k=0}^{\infty}\left(\otimes^{k} \mathcal{H}\right)
$$
where $\otimes^{k} \mathcal{H}$ is the $k$th tensor power of $\mathcal{H}$ for $k \geq 1$ and $\otimes^{0} \mathcal{H}=\mathbb{C} .$
We define the projection on the Full Fock space over $\mathcal{H}$ by
$$P_{+}\left(u_{1} \otimes \cdots \otimes u_{k}\right)=\frac{1}{k !} \sum_{\sigma} u_{\sigma(1)} \otimes \cdots \otimes u_{\sigma(k)},$$
where $\sigma$ ranges over the group of permutations of $k$ letters.
The Bose-Fock space $\mathcal{F}_{+}(\mathcal{H})$ consists of all symmetric tensors, that is to say
$$\mathcal{F}_{+}(\mathcal{H})=P_{+}\bigoplus_{k=0}^{\infty}\left(\otimes^{k} \mathcal{H}\right). $$
 It is easy to verify that if $\left\{h_{j}\right\}$ is an orthonormal basis for $\mathcal{H},$ then
$$\left\{E_{\alpha}=\sqrt{\frac{k !}{\alpha !}} P_{+}\left(h_{1}^{\alpha_{1}} \otimes h_{2}^{\alpha_{2}} \otimes \cdots\right): \quad \alpha=(\alpha_1,\cdots,\alpha_n,\cdots),\sum\alpha_{j}=k, \quad k=0,1,2, \ldots\right\}$$
is an orthonormal basis for $\mathcal{F}_{+}(\mathcal{H}),$ where the superscripts $\alpha_{j}$ denote tensor powers. For any $h\in \mathcal{H}$,
the annihilation operator $a(h)$ and the creation operator $a^*(h)$  on the Fock space over $\mathcal{H}$ are given by
\begin{align*}
&a(h)\left(h_{1} \otimes h_{2} \otimes \cdots \otimes h_{n}\right)=n^{1 / 2}\left(h, h_{1}\right) h_{2} \otimes h_{3} \otimes \cdots \otimes h_{n}\\
&a^{*}(h)\left(h_{1} \otimes h_{2} \otimes \cdots \otimes h_{n}\right)=(n+1)^{1 / 2} h \otimes h_{1} \otimes \cdots \otimes h_{n}.
\end{align*}
By definition, we know that the creation operators map $\otimes^{n} \mathcal{H}$ to $\otimes^{n+1} \mathcal{H}$ and the annihilation operators map $\otimes^{n} \mathcal{H}$ to $\otimes^{n-1} \mathcal{H}$.
The annihilation operator $a_{+}(h)$ and the creation operator $a_{+}^*(h)$  on the Bose-Fock space are given by
$$a_{+}(f)=P_{+} a(f) P_{+}\text{\quad and \quad} a_{+}^{*}(f)=P_{+} a^{*}(f) P_{+},$$
then we have
\begin{equation}\label{creationproperty}
P_{+}(h_1\otimes h_{2} \otimes \cdots h_n)=(n!)^{-1/2}a_{+}^*(h_1)a_{+}^*(h_2)\cdots a_{+}^*(h_n)\Omega,
\end{equation}
where $\Omega=(1,0,0,\cdots)\in \otimes^{n} \mathcal{H}$.
Next, we introduce the Weyl operators on the Bose-Fock space. Let
$$W(h)=\exp \{i \Phi(h)\}$$
be the Weyl operator, where
$$\Phi(h)=2^{-1 / 2} \overline{\left(a_{+}(h)+a_{+}^{*}(h)\right)}.$$
The Weyl algebra $\mathrm{CCR}(\mathcal{H})$  is a $C^*$-algebra generated by
$$\{W(h): h\in \mathcal{H}\},$$
where $\mathrm{CCR}$ stands for the canonical commutation relations, see \cite[pg 10]{Bartteli}.
If $H$ is an unbounded selfadjoint operator on $\mathcal{H}$, one can define $H_{n}$ on $P_{+}\left(\otimes^{k} \mathcal{H}\right)$ by setting $H_{0}=I$ and
$$
H_{n}\left(P_{\pm}\left(f_{1} \otimes \cdots \otimes f_{n}\right)\right)=P_{\pm}\left(\sum_{i=1}^{n} f_{1} \otimes f_{2} \otimes \cdots \otimes H f_{i} \otimes \cdots \otimes f_{n}\right)
$$
for all $f_{i} \in D(H),$ and then extending by continuity. The selfadjoint closure of this sum is called the second quantization of $H$ and is denoted by $d \Gamma(H) .$ Thus
$$
d \Gamma(H)=\overline{\bigoplus_{n \geq 0} H_{n}}.
$$

Let Gaussian measure $d\lambda_n$ on $\mathbb{C}^n$ be given by
$$
d\lambda_n(z)=\frac{1}{{(2\pi)}^{n}}e^{-\frac{|z|^2}{2}}dz.
$$
The Fock space on $\mathbb{C}^n$, denoted by $F^{2}(\mathbb{C}^n,d\lambda_{n})$ or $F^{2}(\mathbb{C}^n)$, consists of all entire functions on $\mathbb{C}^n$ which are square-integrable with respect to $d\lambda_n$.

For any nonnegative integer $k$, let
$$
e_{k}(w)=\sqrt{\frac{1}{2^{k}k !}} w^{k}, w\in\mathbb{C}.
$$
We have $e_1(z)e_k(z)=\sqrt{k+1}e_{k+1}(z)$.
Then the set $\left\{e_{k}\right\}$ is an orthonormal basis for $F^{2}(\mathbb{C},d\lambda_{1})$.

The Gaussian measure can be extended on $\mathbb{C}^\infty$, we denote it by $d\lambda_{\infty}$.
$L^2(\mathbb{C}^\infty,d\lambda_{\infty})$ consists of all square-integrable function on $\mathbb{C}^\infty$ with respect to the infinite dimensional Gaussian measure $d\lambda_{\infty}$. Let $l_j$ be the complex linear functional such that for any $z=(z_1,\cdots,z_j,\cdots)$, we have
$$l_j(z)=z_j.$$
Let $e_{k}\circ l_j$ be a function on $\mathbb{C}^\infty$, such that
$$e_{k}\circ l_j(z)=e_{k}(z_j)=\sqrt{\frac{1}{2^{k}k !}} z_{j}^{k}.$$

The Fock space on $\mathbb{C}^{\infty}$ is defined to be a closed subspace of $L^2(\mathbb{C}^\infty,d\lambda_{\infty})$ generated by the orthonormal set
$$\left\{e_{\alpha_1}\circ l_1\times e_{\alpha_2}\circ l_2\times\cdots\times e_{\alpha_m}\circ l_m\cdots: \quad \sum_{m} \alpha_{m}=k, \quad k=0,1,2, \ldots\right\}$$
and is denoted by $F^2(\mathbb{C}^\infty,d\lambda_{\infty})$.
In fact, $F^2(\mathbb{C}^n)$ can be regarded as a closed subspace of $F^2(\mathbb{C}^\infty,d\lambda_{\infty})$, the embedding is given by
$$f(z_1,\cdots,z_n)\rightarrow f(z_1,\cdots,z_n,0,0,\cdots)\text{, for any }f\in F^2(\mathbb{C}^n). $$
In fact, we have a sequence of embeddings
$$F^2(\mathbb{C}^1)\subset \cdots\subset F^2(\mathbb{C}^n)\subset\cdots\subset F^2(\mathbb{C}^\infty,d\lambda_{\infty}).$$
Let $P_n$ be the projection from $F^2(\mathbb{C}^\infty,d\lambda_{\infty})$ to $F^2(\mathbb{C}^n)$. Since the set of finite polynomials is dense in the Fock space on $\mathbb{C}^{\infty}$, we have
$$\bigcup_{n}F^2(\mathbb{C}^n,d\lambda_{n})\text{ is dense in }F^2(\mathbb{C}^{\infty},d\lambda_{\infty}).$$

We define an isomorphism $B$ from $\mathcal{F}_{+}(\mathcal{H})$ to $F^2(\mathbb{C}^\infty,d\lambda_{\infty})$ such that
$$ B c=c\text{ when }c\in \otimes^{0}\mathcal{H}=\mathbb{C}$$
and
\begin{equation}\label{isomorophism}
B \Big[\sqrt{\frac{k !}{\alpha !}} P_{+}\left(h_{1}^{\alpha_{1}} \otimes \cdots \otimes h_{n}^{\alpha_{n}}\right)\Big]
=e_{\alpha_1}\circ l_1\cdots e_{\alpha_n}\circ l_n
\end{equation}
when $(h_{1}^{\alpha_{1}} \otimes \cdots \otimes h_{n}^{\alpha_{n}})\in \otimes^{k}\mathcal{H}$ with $\sum_{m=1}^{n}\alpha_{m}=k\neq 0$. We need to point out that the isomorphism $B$ depends on the basis $\{h_{j}\}$.
We call $B$ as infinite Bargmann representation.

Let $P$ denote the projection from $L^2(\mathbb{C}^\infty,d\lambda_{\infty})$ to $F^2(\mathbb{C}^\infty,d\lambda_{\infty})$.
We define the Toepltz operator with symbol $\phi\in L^{2}(\mathbb{C}^{\infty},d\lambda_{\infty})$ by
$$T_{\phi}h=P(\phi h)$$
for all $h\in F^{2}(\mathbb{C}^{\infty},d\lambda_{\infty})$ such that $\phi h\in L^{2}(\mathbb{C}^{\infty},d\lambda_{\infty})$. In Section 2, we are going to use some facts of infinite dimensional Gaussian measures to show that the Toeplitz operators on $F^{2}(\mathbb{C}^{\infty},d\lambda_{\infty})$ are unitary equivalent to the annihilation, creation and Weyl operators on the Bose-Fock space, which is the generalization of some conclusions in \cite{coburn1985}. This equivalence can be used to translate some problems in the Bose-Fock space to $F^{2}(\mathbb{C}^{\infty},d\lambda_{\infty})$. Since $\bigcup_{n}F^2(\mathbb{C}^n,d\lambda_{n})\text{ is dense in }F^2(\mathbb{C}^{\infty},d\lambda_{\infty})$, the problems in $F^{2}(\mathbb{C}^{\infty},d\lambda_{\infty})$ can be reduced to  $F^2(\mathbb{C}^n,d\lambda_{n})$. In Section 3, we will use this idea to discuss the a problem in Quantum Statistical Mechanics.

In Section 3, we will study the trace formula which will be applied to the Gibbs state. The Gibbs state of an operator $A$ on the the Bose-Fock space is $$\omega(A)=\frac{\operatorname{Tr}\left(e^{-\beta d \Gamma(H-\mu I)} A\right)}{\operatorname{Tr}\left(e^{-\beta d \Gamma(H-\mu I)}\right)},$$
where $e^{-\beta d \Gamma(H-\mu I)}$ is an operator and its definition will be given in Section \ref{gibbs state}.
The Gibbs state is an important quantity in the Quantum Statistical Mechanics, in fact, it is just the trace of an operator on the Bose-Fock space, see \cite{Bartteli}. Because we know the trace formula in the Fock space on $\mathbb{C}^n$, we can generalize the trace formula and apply it to the Gibbs state.
In the theory of Many-Body Problems, some operators can be represented by the linear combination of products of annihilation and creation operators, see \cite[Chapter 1]{Franz}.
So, it is important to study the product of annihilation and creation operators.
Given $$f^{(1)},\cdots, f^{(m)}, g^{(1)},\cdots, g^{(m)}\in \mathcal{H},$$ we will study the Gibbs state of the operator
$$a_{+}^*(f^{(1)})\cdots a_{+}^*(f^{(m)}) a_{+}(g^{(1)})\cdots a_{+}(g^{(m)}),$$
which is the generalization of \cite[Proposition 5.2.28]{Bartteli}.

In Section 4, we will discuss the relationship between the Fock space on $\mathbb{C}^n$ and the Gaussian Harmonic analysis. In \cite{wu}, the authors gave an isomorphism between the Fock-Sobolev space on $\mathbb{C}^n$ and the Gauss-Sobolev space over $\mathbb{R}^n$. We will generalize this isomorphism in the infinite dimensional case. As an application, we study the boundedness of the annihilation and creation operators.

\section{Annihilation, creation and Weyl operators}
Let $\chi_{k}$ be the subspace of the Fock space on $\mathbb{C}^n$ generated by
$$\left\{e_{\alpha_1}\circ l_1\cdots e_{\alpha_m}\circ l_m\cdots: \quad \sum_{j} \alpha_{j}=k\right\}.$$
Then, for any $h=\sum c_jh_j\in \mathcal{H},$
we have $Bh=\sum_{j} c_je_1\circ l_j.$
By $P_{+}h_m=h_m$ and (\ref{isomorophism}), we have
$$Bh_m(z)=e_1(z_m).$$
Thus, we know that the isomorphism $B$ maps $\mathcal{H}$ to $\chi_{1}$.

Next, we are going to show that the annihilation operators and creation operators are isomorphic to the Toeplitz operators with symbols in $\chi_{1}$.
\begin{proposition}\label{creation}
For any $h\in\mathcal{H}$, let $a_{+}(h)$ and $a_{+}^*(h)$ be an annihilation operator and a creation operator, we have
$$Ba_{+}^*(h)B^{-1}=T_{Bh}\text{\quad and\quad }Ba_{+}(h)B^{-1}=T_{\overline{Bh}}.$$
\end{proposition}
\begin{proof}
For any $h=\sum c_jh_j\in \mathcal{H},$
we have
$$Bh=\sum c_je_1\circ l_j.$$
Thus, for any $e_{\alpha_1}\circ l_1\cdots e_{\alpha_m}\circ l_m\cdots$ with $\sum \alpha_{j}=k$, we have
\begin{align*}
&B^{-1}T_{Bh}e_{\alpha_1}\circ l_1\cdots e_{\alpha_m}\circ l_m\cdots\\
=&B^{-1}\sum_{j} c_je_1\circ l_j\times e_{\alpha_1}\circ l_1\cdots e_{\alpha_m}\circ l_m\cdots\\
=&B^{-1}\sum_{j} c_j \sqrt{\alpha_{j}+1} e_{\alpha_1}\circ l_1\cdots e_{\alpha_j+1}\circ l_j\cdots e_{\alpha_m}\circ l_m\cdots\\
=&\sum_{j} c_j \sqrt{\alpha_{j+1}} \sqrt{\frac{(k+1) !}{\alpha_1! \cdots (\alpha_{j}+1)! \cdots}} P_{+}\left(h_{1}^{\alpha_{1}}\otimes\cdots \otimes h_{j}^{\alpha_{j}+1} \otimes \cdots\right)\\
=&\sum_{j} c_j  \sqrt{\frac{(k+1) !}{\alpha!}} P_{+}\left(h_{1}^{\alpha_{1}}\otimes\cdots \otimes h_{j}^{\alpha_{j}+1} \otimes \cdots\right).
\end{align*}
On the other hand
\begin{align*}
a_{+}^*(h)B^{-1}e_{\alpha_1}\circ l_1\cdots e_{\alpha_m}\circ l_m\cdots
&=a_{+}^*(h)\sqrt{\frac{k !}{\alpha !}} P_{+}\left(h_{1}^{\alpha_{1}} \otimes h_{2}^{\alpha_{2}} \otimes \cdots\right)\\
&=\sqrt{\frac{1}{\alpha !}}[a_{+}^*(h)a_{+}^*(h_1)^{\alpha_1}\cdots a_{+}^*(h_n)^{\alpha_n}\cdots]\Omega\text{ \big(by (\ref{creationproperty})\big)}\\
&=\sqrt{\frac{(k+1)!}{\alpha !}}P_{+}(h\otimes h_1^{\alpha_1}\otimes \cdots h_n^{\alpha_n}\cdots )\\
&=\sqrt{\frac{(k+1)!}{\alpha !}}\sum c_j P_{+}(h_j\otimes h_1^{\alpha_1}\otimes \cdots h_n^{\alpha_n}\cdots ),\\
\end{align*}
which means that $Ba_{+}^*(h)B^{-1}=T_{Bh}$. Thus we have $Ba_{+}(h)B^{-1}=T_{\overline{Bh}}$ by taking adjoints.
\end{proof}

We are going to give a representation of the Weyl operators, thus we need some facts about infinite dimensional Gaussian measure.
For these facts about infinite dimensional Gaussian measure, we refer to \cite[Chapter 2]{Bogachev}. Because we need a particular theorem in \cite{Bogachev}, we give the details about the general theory.

A Borel probability measure $\gamma$ on $\mathbb{R}^{1}$ is called Gaussian if it is either the Dirac measure $\delta_{a}$ at a point $a$ or has density
$$
p\left(\cdot, a ,\sigma^{2}\right): t \mapsto \frac{1}{\sigma \sqrt{2 \pi}} \exp \left(-\frac{(t-a)^{2}}{2 \sigma^{2}}\right)
$$
with respect to the Lebesgue measure.

Let $X$ be a locally convex space. Let $X^{*}$ be the set of real linear continuous functionals on $X$.
Let us denote by $\mathcal{E}(X, X^{*})$ the minimal $\sigma$-field of subsets of $\Omega,$ with respect to which all functionals $f \in X^{*}$ are measurable.
A probability measure $\gamma$ defined on the $\sigma-$field $\mathcal{E}(X, X^{*})$ is called Gaussian if, for any $f \in X^{*},$ the induced measure  $\gamma \circ f^{-1}$ on $\mathbb{R}^{1}$ is Gaussian. Let
$$a_{\gamma}(f)=\int_{X} f(x) \gamma(d x).$$
We denote by $X_{\gamma}^{*}$ the closure of the set
$$
\left\{f-a_{\gamma}(f), f \in X^{*}\right\}
$$
embedded into $L^{2}(\gamma),$ with respect to the norm of $L^{2}(\gamma) .$
We define
$$
\langle f,g\rangle_{X_{\gamma}^{*}}:=\int_{X}\left[f(x)-a_{\gamma}(f)\right]\left[g(x)-a_{\gamma}(g)\right] \gamma(d x)
$$
For any $x\in X$, let
$$|h|_{ X(\gamma)}=\sup \left\{l(h): l \in X^{*}, \langle l,l\rangle_{X_{\gamma}^{*}} \leq 1\right\}.$$
The space
$$ X(\gamma)=\left\{h \in X:|h|_{X(\gamma)}<\infty\right\}$$
is called the Cameron-Martin space for $X$.

\begin{lemma}[\!\!\protect{\cite[Lemma 2.4.1]{Bogachev}}]\label{represantation}
$A$ vector $x$ in $X$ belongs to the Cameron-Martin space $X(\gamma)$ precisely when there exists $\hat{x} \in X_{\gamma}^{*}$  such that
$f(x)=\langle \hat{x},f\rangle_{X_{\gamma}^{*}} $ for any $f\in X^{*}$.
In this case,
$$
|x|_{X(\gamma)}=\|\hat{x}\|_{L^{2}(\gamma)}.
$$
\end{lemma}
\begin{lemma}[\!\!\protect{\cite[Lemma 2.4.4]{Bogachev}}]\label{translation}
Let $\gamma$ be a Gaussian measure on a locally convex space $X$. If $x\in X(\gamma)$ and $\hat{x} \in X_{\gamma}^{*}$ satisfy
 $$f(x)=\langle \hat{x},f\rangle_{X_{\gamma}^{*}}$$
for any $f\in X^{*},$ then the measures $\gamma$ and $\gamma_{x}=\gamma(\cdot-x)$ are equivalent and the corresponding Radon-Nikodym density is given by the expression
$$
d\gamma_{x}(z)=\exp \left(\hat{x}(z)-\frac{1}{2}|x|_{X(\lambda)}^{2}\right)d\gamma(z).
$$
\end{lemma}
Next, we discuss a special case when $X=\mathbb{C}^{\infty}$ and $d\gamma=d\lambda_{\infty}$.
Let $\Re l_k(z)=\Re z_k$ and $\Im l_k(z)=\Im z_k.$
Then any real linear functional $f\in(\mathbb{C}^\infty)^{*}$ is a finite linear combination of $\Re l_k$ and $\Im l_k$, see \cite[Theorem 4.3]{Schaefer}. That is to say that
there is a unique sequence $\{f_k=a_k+ib_k\in \mathbb{C}\}$ such that
$$f=\sum_{k=1}^{n}a_k \Re l_k+\sum_{k=1}^{n}b_k \Im l_k=\sum_{k=1}^{n}\Re f_k \Re l_k+\sum_{k=1}^{n}\Im f_k \Im l_k=\Re \sum_{k=1}^{n}\overline{f_k} l_k.$$
Thus, we have
$$a_{d\lambda_{\infty}}(f)=\int_{\mathbb{C}^\infty} f(z) d\lambda_{\infty}(z)=0.$$
For any
$$f=\Re\sum_{k=1}^{n}\overline{f_k} l_k\text{\quad and\quad} g=\Re\sum_{k=1}^{n}\overline{g_k} l_k\in (\mathbb{C}^\infty)^{*},$$
we have
\begin{equation}\label{finite norm}
\langle f,g\rangle_{(\mathbb{C}^{\infty})^{*}_{d\lambda_{\infty}}}=
\int_{\mathbb{C}^{\infty}}f(z)g(z) d\lambda_{\infty}(z)=\sum_{k=1}^{n}(\Re f_k\Re g_k+\Im f_k\Im g_k)=\Re\sum_{k=1}^{n} \overline{f_k} g_k.
\end{equation}
Since $(\mathbb{C}^{\infty})^{*}_{d\lambda_{\infty}}$ is the completion of $(\mathbb{C}^\infty)^{*}$ with respect to the inner product above, we know that, for any $f\in(\mathbb{C}^{\infty})^{*}_{d\lambda_{\infty}}$, there is a sequence $\{f_k\}$ such that
$$f=\lim_{n\rightarrow\infty}\Re \sum_{k=1}^{n}\overline{f_k }l_k,$$
where the limit is taken in $(\mathbb{C}^{\infty})^{*}_{d\lambda_{\infty}}$. We denote by $f=\Re \sum_{k=1}^{\infty}\overline{f_k }l_k$.

If $z$ is in the Cameron-Martin space $\mathbb{C}^{\infty}(d\lambda_{\infty})$ for $\mathbb{C}^{\infty}$, then we have
$$|z|_{ \mathbb{C}^{\infty}(\gamma)}=\sup \left\{l(z): l \in (\mathbb{C}^{\infty})^{*}, \langle l,l\rangle_{(\mathbb{C}^{\infty})^{*}_{d\lambda_{\infty}}} \leq 1\right\}<\infty.$$
Thus, by (\ref{finite norm}), we have
$$|z|_{ \mathbb{C}^{\infty}(\gamma)}=(\sum_{k=1}^{\infty}|z_k|^2)^{1/2}<\infty. $$
Thus, we write $|z|_{ \mathbb{C}^{\infty}(\gamma)}$ as $|z|$.
By (\ref{finite norm}), we know that $\Re\sum_{k=1}^{n}\overline{z_{k}}l_k$ is convergent in $L^2(\mathbb{C}^{\infty},d\lambda_{\infty})$,
then we denote the limit by
\begin{equation}\label{realinnerfunction}
\hat{z}=\Re \sum_{k=1}^{\infty}\overline{z_{k}}l_k.
\end{equation}
Then for any $f=\Re \sum_{k=1}^{n}\overline{f_k }l_k$, we have
$$f(z)=\Re\sum_{k=1}^{n}\overline{f_k}z_k=\Re\langle f,\hat{z}\rangle_{(\mathbb{C}^{\infty})^{*}_{d\lambda_{\infty}}}.$$

By Lemma \ref{represantation}, Lemma \ref{translation} and the argument above, we have the following Lemma.
\begin{lemma}\label{translation2}
Let $d\lambda_{\infty}$ be the Gaussian measure on $\mathbb{C}^{\infty}$.\\
(1) Let $\mathbb{C}^{\infty}(d\lambda_{\infty})$ be the Cameron-Martin space for $\mathbb{C}^{\infty}$, then
$$x=(x_1,\cdots,x_{k},\cdots)\in \mathbb{C}^{\infty}(d\lambda_{\infty})\text{\quad if and only if \quad}|x|=(\sum_{k=1}^{\infty}|x_k|)^{1/2}<\infty.$$
(2) For any $f,g\in(\mathbb{C}^{\infty})^{*}_{d\lambda_{\infty}}$, there are sequences $(f_1,\cdots,f_n,\cdots)$ and $(g_1,\cdots,g_n,\cdots)$ with
$(\sum_{k=1}^{\infty}|f_k|)^{1/2}<\infty$ and $(\sum_{k=1}^{\infty}|g_k|)^{1/2}<\infty$ such that
$$f=\Re\sum_{k=1}^{\infty}\overline{ f_{k}}l_k,\quad g=\Re\sum_{k=1}^{\infty}\overline{ g_{k}}l_k\text{\quad and \quad}
\langle g,f\rangle_{(\mathbb{C}^{\infty})^{*}_{d\lambda_{\infty}}}=\Re\sum_{k=1}^{\infty}\overline{ f_{k}}g_k.$$
(3) For any $x=(x_1,\cdots,x_{k},\cdots)\in \mathbb{C}^{\infty}(d\lambda_{\infty})$, let $\hat{x}=\Re\sum_{k}\overline{x_k}l_{k}\in (\mathbb{C}^{\infty})_{d\lambda_{\infty}}^{*}$, we have
$$f(x)=\langle \hat{x},f\rangle,$$
for any $f\in(\mathbb{C}^{\infty})^{*}$.
Moreover the measures $d\lambda_{\infty}$ and $d\lambda_{\infty,x}=d\lambda_{\infty}(\cdot-x)$ are equivalent and the corresponding Radon-Nikodym density is given by the expression
$$
d\lambda_{\infty,x}(z)=\exp \left(\hat{x}(z)-\frac{1}{2}|x|_{\mathbb{C}^{\infty}(d\lambda_{\infty})}^{2}\right)d\lambda_{\infty}(z).
$$
\end{lemma}

For any $x=(x_1,\cdots,x_n,\cdots)\in \mathbb{C}^{\infty}(d\lambda_{\infty})$, we have $\sum_{k=1}^{n}\overline{x_k}l_k$ is convergent in $F^2(\mathbb{C}^{\infty},d\lambda_{\infty})$. Let $\lim_{n\rightarrow\infty}\sum_{k=1}^{n}\overline{x_k}l_k$ be the limit, we define
\begin{equation}\label{innerproduct}
\langle x,z\rangle=\Big[\lim_{n\rightarrow\infty}\sum_{k=1}^{n}\overline{x_k}l_k\Big](z)
\text{ for } z\in\mathbb{C}^{\infty}\text{ almost everywhere}.
\end{equation}
Thus $\langle x,z\rangle$ is a function in $F^2(\mathbb{C}^{\infty},d\lambda_{\infty})$.
By (\ref{finite norm}) and (\ref{realinnerfunction}), we have
\begin{equation}\label{infinite norm}
\Re\langle x,z\rangle=\hat{x}(z)\text{ almost everywhere and }\|\langle x,z\rangle\|_{(\mathbb{C}^{\infty})_{d\lambda_{\infty}}^{*}}=2|x|=2\|\hat{x}\|_{\mathbb{C}^{\infty}(d\lambda_{\infty})}.
\end{equation}
We define the translation operator $U_x$ on $L^2(\mathbb{C}^{\infty},d\lambda_{\infty})$ by
$$U_xf(z)=f(z-x)e^{\frac{1}{2}\langle x,z\rangle-\frac{1}{4}|x|^{2}}.$$
\begin{remark} In the Fock space on $\mathbb{C}^n$, we can define the reproducing kernel for all $x^{(n)}=\{x_1,\cdots,x_n\}\in\mathbb{C}^{n}$. The reproducing kernel is given by
$$K_{x^{(n)}}(z^{(n)})=e^{\frac{1}{2}\sum_{j}^{n}\overline{x_j}z_j}=e^{\frac{1}{2}\langle x^{(n)},z^{(n)}\rangle_{\mathbb{C}^n}}$$
for any $z^{(n)}\in \mathbb{C}^n$. Since we can't define the inner product of two points in $\mathbb{C}^{\infty}$, we can't define the reproducing kernel for all points in $\mathbb{C}^{\infty}$.
However, for $x$ in the Cameron-Martin space $\mathbb{C}^{\infty}(d\lambda_{\infty})$ and $z\in \mathbb{C}^{\infty}$, we can define $\langle x, z\rangle$ by the limit in $F^2(\mathbb{C}^{\infty},d\lambda_{\infty})$ as in (\ref{innerproduct}). That is the reason that we need the Cameron-Martin space.
So, the analogue of the normalized reproducing kernel in $F^2(\mathbb{C}^{\infty},d\lambda_{\infty})$ is given by
$$k_x(z)=e^{\frac{1}{2}\langle x,z\rangle-\frac{1}{4}|x|^{2}}.$$
On the other hand, by Lemma \ref{translation2}, we know that any $x\in\mathbb{C}^{\infty}(d\lambda_{\infty})$ is a bounded sequence. Thus
$$\mathbb{C}^{\infty}(d\lambda_{\infty})\subset \bigcup_{N} \Big[B(0,N)\times\cdots\times B(0,N)\times\cdots\Big]\subset \mathbb{C}^{\infty},$$
where $B(0,N)$ is a ball in $\mathbb{C}$ with center $0$ and radius $N$. Since $\Big[B(0,N)\times\cdots\times B(0,N)\times\cdots\Big]$ has measure $0$ in $\mathbb{C}^{\infty}$ with respect to the Gaussian measure, we know that $\mathbb{C}^{\infty}(d\lambda_{\infty})$ has measure $0$.

If $x, z\in\mathbb{C}^{\infty}(d\lambda_{\infty})$, we always require
$$\langle x,z\rangle=\lim_{n\rightarrow\infty}\sum_{k=1}^{n}\overline{x_k}z_k.$$
Since $\mathbb{C}^{\infty}(d\lambda_{\infty})$ has measure $0$, this requirement would't change the definition of $\langle x,z\rangle$ as a function in $F^2(\mathbb{C}^{\infty},d\lambda_{\infty})$.
\end{remark}

We are going to study the operators $BW(h)B^{-1}$, so we need a lemma.
\begin{lemma}\label{weyloperator}
For any $x,y\in \mathbb{C}^{\infty}(d\lambda_{\infty})$, we have
$$U_yU_x=e^{-\frac{i}{2}\Im\langle x,y\rangle}U_{x+y}f(z) \text{\quad and \quad}
\|U_xf\|^{2}_{L^2(\mathbb{C}^{\infty},d\lambda_{\infty})}=\|f\|^{2}_{L^2(\mathbb{C}^{\infty},d\lambda_{\infty})}$$
for any $f\in L^{2}(\mathbb{C}^{\infty},d\lambda_{\infty})$.
If $x^{(m)}$ converges to $x$ in $\mathbb{C}^{\infty}(d\lambda_{\infty})$, then $U_{x^{(m)}}$ converges to $U_x$ in the strong operator topology on $F^2(\mathbb{C}^{\infty},d\lambda_{\infty})$.
Moreover, $U_{x}$ is an unitary operator on the Fock space over $\mathbb{C}^{\infty}$, thus $U_{x}f\in F^2(\mathbb{C}^{\infty},d\lambda_{\infty})$ when $f\in F^2(\mathbb{C}^{\infty},d\lambda_{\infty})$.
\end{lemma}
\begin{proof}For any $f\in{L^2(\mathbb{C}^{\infty},d\lambda_{\infty})}$, by Lemma \ref{translation2}, we have
\begin{align*}
\|f\|^2_{L^2(\mathbb{C}^{\infty},d\lambda_{\infty})}&=\int_{\mathbb{C}^{\infty}}|f(z)|^2d\lambda_{\infty}\\
&=\int_{\mathbb{C}^{\infty}}|f(z-x)|^2d\lambda_{\infty,x}\\
&=\int_{\mathbb{C}^{\infty}}|f(z-x)|^2\exp \left(\hat{x}(z)-\frac{1}{2}|x|_{\mathbb{C}^{\infty}(d\lambda_{\infty})}^{2}\right)d\lambda_{\infty}(z)\\
&=\int_{\mathbb{C}^{\infty}}|f(z-x)|^2\exp \left(\Re\langle x,z\rangle-\frac{1}{2}|x|_{\mathbb{C}^{\infty}(d\lambda_{\infty})}^{2}\right)d\lambda_{\infty}(z)\\
&=\int_{\mathbb{C}^{\infty}}|f(z-x)\exp \left(\frac{1}{2}\langle x,z\rangle-\frac{1}{4}|x|_{\mathbb{C}^{\infty}(d\lambda_{\infty})}^{2}\right)|^2 d\lambda_{\infty}(z)\\
&=\int_{\mathbb{C}^{\infty}}|U_xf(z)|^2 d\lambda_{\infty}(z).
\end{align*}
Thus $U_{x}$ is a unitary operator on $L^2(\mathbb{C}^{\infty},d\lambda_{\infty})$.
For any $x,y\in \mathbb{C}^{\infty}(d\lambda_{\infty})$, we have
\begin{align*}
U_yU_xf(z)&=U_y[f(z-x)e^{\frac{1}{2}\langle x,z\rangle-\frac{1}{4}|x|^{2}}]\\
&=f(z-x-y)e^{\frac{1}{2}\langle x,z-y\rangle-\frac{1}{4}|x|^{2}}e^{\frac{1}{2}\langle y,z\rangle-\frac{1}{4}|y|^{2}}\\
&=e^{-\frac{i}{2}\Im\langle x,y\rangle}U_{x+y}f(z).
\end{align*}
Thus $U_{-x}U_x=e^{-\frac{i}{2}\Im\langle x,x\rangle}U_{x-x}=I$, that is to say $U_{-x}=U^{-1}_{x}=U^{*}_{x}.$

For any $x^{(m)},x\in \mathbb{C}^{\infty}(d\lambda_{\infty})$ and $f\in F^2(\mathbb{C}^{\infty},d\lambda_{\infty})$, if $x^{(n)}$ converges to $x$ in $\mathbb{C}^{\infty}(d\lambda_{\infty})$, then
\begin{align*}
&\|U_{x^{(m)}}f-U_{x}f\|_{F^2(\mathbb{C}^{\infty},d\lambda_{\infty})}\\
=&\|e^{-\frac{i}{2}\Im\langle x^{(m)},x\rangle}U_{x^{(m)}-x}f-f\|_{F^2(\mathbb{C}^{\infty},d\lambda_{\infty})}\\
\leq& \|(e^{-\frac{i}{2}\Im\langle x^{(m)},x\rangle}-1)U_{x^{(m)}-x}f\|_{F^2(\mathbb{C}^{\infty},d\lambda_{\infty})}+
\|U_{x^{(m)}-x}f-f\|_{F^2(\mathbb{C}^{\infty},d\lambda_{\infty})}\\
\leq& |e^{-\frac{i}{2}\Im\langle x^{(m)},x\rangle}-1|\|f\|_{F^2(\mathbb{C}^{\infty},d\lambda_{\infty})}+
\|U_{x^{(m)}-x}f-f\|_{F^2(\mathbb{C}^{\infty},d\lambda_{\infty})}.
\end{align*}
On one hand,
$$
\lim_{m\rightarrow\infty}|\Im\langle x^{(m)},x\rangle|=\lim_{m\rightarrow\infty}|\Im\langle x^{(m)}-x,x\rangle|
\leq\lim_{m\rightarrow\infty}| x^{(m)}-x||x|=0.
$$
On the other hand, we claim that if $y^{(m)}\rightarrow0$ in $\mathbb{C}^{\infty}(d\lambda_{\infty})$, then
$$\lim_{m\rightarrow\infty}\|U_{y^{(m)}}f-f\|_{F^2(\mathbb{C}^{\infty},d\lambda_{\infty})}=0.$$
Thus, we obtain
$$\lim_{m\rightarrow\infty}\|U_{x^{(m)}}f-U_{x}f\|_{F^2(\mathbb{C}^{\infty},d\lambda_{\infty})}=0.$$
Next, we prove the claim.
If there is a constant $\epsilon>0$, a function $g\in F^2(\mathbb{C}^{\infty},d\lambda_{\infty})$ and a sequence ${y^{(m)}}$ such that
$$\lim_{m\rightarrow\infty}|y^{(m)}|=0 \text{\quad and \quad} \|U_{y^{(m)}}g-g\|_{F^2(\mathbb{C}^{\infty},d\lambda_{\infty})}>3\epsilon$$
for any $m$. Since $g\in F^2(\mathbb{C}^{\infty},d\lambda_{\infty})$ there is a polynomial $p(z_1,\cdots,z_j)$ such that
$$\|g-p\|_{F^2(\mathbb{C}^{\infty},d\lambda_{\infty})}<\epsilon.$$
Thus, we have
$$\|U_{y^{(m)}}p-p\|_{F^2(\mathbb{C}^{\infty},d\lambda_{\infty})}>\epsilon.$$
By the expression of the $U_{y^{(m)}}$ and $(\ref{infinite norm})$ we have
$$U_{y^{(m)}}p(z)=p(z_{1}-y^{(m)}_{1},\cdots,z_{j}-y^{(m)}_{j})e^{\frac{1}{2}\langle y^{m},z\rangle-\frac{1}{4}|y^{(m)}|^{2}}
\text{\quad and\quad}
\|\langle y^{m},z\rangle\|_{(\mathbb{C}^{\infty})^{*}_{d\lambda_{\infty}}}=2|y^{(m)}|\rightarrow0.
$$
Thus, there is a subsequence $\{y^{(m_k)}\}$ such that $\langle y^{m_{k}},z\rangle$ converges to $0$ almost everywhere.
Then $U_{y^{(m_k)}}p$ converges to $p$ almost everywhere. On the other hand we have
$$\|U_{y^{(m_k)}}p\|_{L^2(\mathbb{C}^{\infty},d\lambda_{\infty})}=\|p\|_{L^2(\mathbb{C}^{\infty},d\lambda_{\infty})},$$
thus
$$\lim_{m\rightarrow\infty}\|U_{y^{(m_k)}}p-p\|_{F^2(\mathbb{C}^{\infty},d\lambda_{\infty})}=0,$$
which is a contradiction. We have completed the proof of the claim.

In the last, we only need to prove that for any polynomial $p(z_1,\cdots,z_j)$, we have
$U_{x}p\in F^2(\mathbb{C}^{\infty},d\lambda_{\infty}).$
For any $l\geq j$, let $x^{(l)}=(x_1,x_2,\cdots,x_l,0,0,\cdots)$, we have
$$U_{x^{(l)}}p(z_1,\cdots,z_j)=p(z_1-x_1,\cdots,z_j-x_j)
e^{\frac{1}{2}\sum_{m=1}^{l}\overline{x_m}z_m+\frac{1}{4}\sum_{m=1}^{l}|x_m|^2}\in F^2(\mathbb{C}^{l},d\lambda_{l}). $$
Thus $U_{x^{(l)}}p\in F^2(\mathbb{C}^{\infty},d\lambda_{\infty})$ for any $l\geq j$. Send $l$ to $\infty$, and we have
$$U_{x}p\in F^2(\mathbb{C}^{\infty},d\lambda_{\infty}).$$
Thus, $U_{x}$ is an unitary operator from $F^2(\mathbb{C}^{\infty},d\lambda_{\infty})$ to $F^2(\mathbb{C}^{\infty},d\lambda_{\infty})$.
\end{proof}
For any $x=(x_1,x_2,\cdots)\in \mathbb{C}^{\infty}(d\lambda_{\infty})$, since $\Re\langle x,z\rangle$ is a real-valued function, we have $T_{\Re\langle x,z\rangle}$ is an unbounded self-adjoint operator. Thus we can define
$e^{i T _{\Re\langle x, z\rangle}}$ by functional calculus.
Let $x^{(n)}=(x_1,\cdots,x_n,0,0,\cdots)$,
then $U_{x^{(n)}}$ is an operator maps $F^2(\mathbb{C}^n,d\lambda_{n})$ to $F^2(\mathbb{C}^n,d\lambda_{n})$. In \cite[Proposition 2 and pg 284]{coburn1985}, the authors have proved that
\begin{equation}\label{three operators}
T_{e^{i\Re\langle x^{(n)}, z\rangle+\frac{|x^{(n)}|^2}{4}}}=e^{i T _{\Re\langle x^{(n)}, z\rangle}}=U_{-i\overline{x^{(n)}}}
\end{equation}
holds on the Fock space on $\mathbb{C}^n$. We can generalize this equality in the Fock space on $\mathbb{C}^{\infty}$.
\begin{theorem}\label{weyl}
For any $h=\sum_{j=1}^{\infty} x_jh_j\in\mathcal{H}$, let $x=(x_1,\cdots,x_n,\cdots)$, we have
$$ BW(h)B^{-1}=e^{i T _{\Re\langle x, z\rangle}}=U_{-i\overline{x}}=T_{e^{i\Re\langle x, z\rangle+\frac{|x|^2}{4}}}$$
holds on the Fock space on $\mathbb{C}^{\infty}$.
Moreover, the Weyl algebra is represented on $F^{2}(\mathbb{C}^{\infty},d\lambda_{\infty})$ as a Toeplitz algebra generated by
$$
\left\{T_{e^{i\Re\langle x, z\rangle+\frac{|x|^2}{4}}}: x\in\mathbb{C}^{\infty}(d\lambda_{\infty})  \right\}.
$$
\end{theorem}
\begin{proof}
Since the Weyl algebra is generated by Weyl operators, we only need to prove the equality.

By the definition of the isomorphism $B$, we have $Bh=\sum_{j=1}^{\infty} x_je_1\circ l_j.$
Since $\sum_{j}|x_j|^2< \infty$, we have $x\in \mathbb{C}^{\infty}(d\lambda_{\infty})$.
Then $\langle x,z \rangle$ is well-defined and
$$\Re \langle x,z \rangle= \Re \sum_{j=1}^{\infty} x_j  l_j=\sqrt{2}\Re \sum_{j=1}^{\infty} x_je_j\circ l_j
=2^{-1/2}(\overline{\sum_{j=1}^{\infty} x_je_j\circ l_j}+\sum_{j=1}^{\infty} x_je_j\circ l_j).$$
Moreover, by Proposition \ref{creation}, we have
 $Ba_{+}^*(h)B^{-1}=T_{Bh}$ and $Ba_{+}(h)B^{-1}=T_{\overline{Bh}}$.
 Thus
$$BW(h)B^{-1}=Be^{i2^{-1 / 2} \overline{(a_{+}(h)+a_{+}^{*}(h))}}B^{-1}=e^{i2^{-1 / 2}
\Big(T_{\overline{\sum_{j=1}^{\infty} x_je_1\circ l_j}}+T_{\sum_{j=1}^{\infty} x_je_1\circ l_j}\Big)}
=e^{i T _{\Re\langle x, z\rangle}}.$$
Let $x^{(m)}=(x_1,\cdots,x_m,0,0,\cdots)$, by (\ref{three operators}), we have
$$T_{e^{i\Re\langle x^{(m)}, z\rangle+\frac{|x^{(m)}|^2}{4}}}f=e^{i T _{\Re\langle x^{(m)}, z\rangle}}f=U_{-i\overline{x^{(m)}}}f$$
for any $m,n\in \mathbb{N}$ with $m\geq n$ and $f\in F^2(\mathbb{C}^{n},d\lambda_{n})\subseteq F^2(\mathbb{C}^{m},d\lambda_{m}).$
For the operator $T_{e^{i\Re\langle x^{(m)}, z\rangle+\frac{|x^{(m)}|^2}{4}}}$, since $\lim_{m\rightarrow\infty}\|x^{(m)}-x\|_{\mathbb{C}^{\infty}(d\lambda_{\infty})}=0$, we have
$$ \lim_{m\rightarrow\infty}\|\langle x^{(m)},z\rangle-\langle x,z\rangle\|_{L^2(\mathbb{C}^{\infty},d\lambda_{\infty})}=0.$$
There is a subsequence $m_{k}$ such that
$$\lim_{k\rightarrow\infty}\langle x^{(m_k)},z\rangle=\langle x,z\rangle\text{\quad almost everywhere}.$$
Thus we have
$$\lim_{k\rightarrow\infty} T_{e^{i\Re\langle x^{(m_k)}, z\rangle+\frac{|x^{(m_k)}|^2}{4}}}f=\lim_{k\rightarrow\infty} T_{e^{i\Re\langle x, z\rangle+\frac{|x|^2}{4}}}f.$$
Let $h^{(m)}=\sum_{j=1}^{m} x_jh_j$, by \cite[Proposition 5.2.4]{Bartteli}, we have $W(h^{(m)})$ converges to $W(h)$ in the strong operator topology.
Thus
$$\lim_{m\rightarrow\infty}e^{i T _{\Re\langle x^{(m)}, z\rangle}}f=e^{i T _{\Re\langle x, z\rangle}}f.$$
By Proposition \ref{weyloperator}, we have
$$\lim_{m\rightarrow\infty}U_{-i\overline{x^{(m)}}}f=U_{-i\overline{x}}f.$$
Then, we have
$$e^{i T _{\Re\langle x, z\rangle}}f=U_{-i\overline{x}}f=T_{e^{i\Re\langle x, z\rangle+\frac{|x|^2}{4}}}f$$
for any $f\in F^2(\mathbb{C}^{n},d\lambda_{n})$.
We have completed the proof because $\bigcup_{n}F^2(\mathbb{C}^{n},d\lambda_{n})$ is dense in $F^2(\mathbb{C}^{\infty},d\lambda_{\infty})$.
\end{proof}

\section{The Gibbs state and trace formula}\label{gibbs state}
Next, we study the second quantization and the Gibbs state. Given an unbounded self-adjoint operator $H$ on the Hilbert space $\mathcal{H}$, we have defined $d \Gamma(H)$, thus $e^{it d\Gamma(H)}$ is an unitary operator for any $t$. Let $U_t=e^{itH}$, then $U_t$ is an unitary operator on $\mathcal{H}$.
By \cite[pg 8]{Bartteli}, we know that
\begin{equation}\label{unitarygroup}
e^{it d\Gamma(H)} P_{+}(f_{1} \otimes  \cdots \otimes f_n)=P_{+}(e^{itH} f_{1} \otimes \cdots \otimes e^{itH}f_n).
\end{equation}
Since $d\Gamma(H)$ is a self-adjoint operator, we know that $e^{it d\Gamma(H)}$ is a bounded operator.

Let $\beta$ be a constant such that $\beta H$ is a positive operator, then $\beta d\Gamma(H)$ is a positive operator. Similarly, the operator $e^{-\beta d\Gamma(H)}$ is given by
\begin{equation}\label{oneparametergroup}
e^{-\beta d\Gamma(H)} P_{+}(f_{1} \otimes\cdots\otimes f_n)=P_{+}(e^{-\beta H} f_{1} \otimes \cdots \otimes e^{-\beta H}f_n).
\end{equation}
Since $\beta d\Gamma(H)$ is a positive operator, we know that $e^{-\beta d\Gamma(H)}$ is a bounded operator.

Because there is no proof of (\ref{unitarygroup}) and (\ref{oneparametergroup}) in \cite{Bartteli}, we give a simple proof here.

Let $U_t$ be an operator such that
$$U_tP_{+}( f_{1} \otimes \cdots \otimes f_n)= P_{+}(e^{itH} f_{1} \otimes \cdots \otimes e^{itH}f_n).$$
It is easy to see that $U_t$ is a one-parameter group of unitary operators and its generator is $d\Gamma(H)$, by \cite[Theorem
$\uppercase\expandafter{\romannumeral8}.8$] {Reed} we have
$$U_t=e^{it d\Gamma(H)}.$$
The proof for (\ref{oneparametergroup}) is similar, we only need to apply \cite[Problem 39, pg 315]{Reed}.

Recall that, for any $j\in\mathbb{N}$, $l_{j}$ is a function in $F^2(\mathbb{C}^{\infty},d\lambda_{\infty})$ such that $l_{j}(z)=z_j$.
Since $B$ is an isomorphism from $\mathcal{H}$ to $\chi_1$,
we know that $Be^{itH}B^{-1}$ is an operator on $\chi_1$ which is generated by $\{l_{j}\}$.
\begin{theorem}\label{second quantization}
For any self-adjoint operator $H$ on $\mathcal{H}$, we have
$$Be^{it d\Gamma(H)}B^{-1}f(z)=f(Be^{itH}B^{-1}l_1(z),\cdots,Be^{itH}B^{-1}l_n(z),\cdots)$$
for any $f\in \bigcup_{n}F^2(\mathbb{C}^{n},d\lambda_{n})$. Let $\beta$ be a constant such that $\beta H$ is a positive operator, then we have
$$Be^{-\beta d\Gamma(H)}B^{-1}f(z)=f(Be^{-\beta H}B^{-1}l_1(z),\cdots,Be^{-\beta H}B^{-1}l_n(z),\cdots)$$
for any $f\in \bigcup_{n}F^2(\mathbb{C}^{n},d\lambda_{n})$.
\end{theorem}
\begin{proof}
We only need to prove the first conclusion, because the proof for the second is similar. Since
\begin{align*}
B \Big[\sqrt{\frac{k !}{\alpha !}} P_{+}\left(h_{1}^{\alpha_{1}} \otimes h_{2}^{\alpha_{2}} \otimes \cdots\otimes h_{n}^{\alpha_{n}}\right)\Big]
&=e_{\alpha_1}\circ l_1\cdots e_{\alpha_n}\circ l_n\\
&=\sqrt{\frac{1}{2^k\alpha!}}l_{1}^{\alpha_1}\cdots l_{n}^{\alpha_n}
\end{align*}
and $Bh_{j}(z)=e_{1}\circ l_j=\sqrt{\frac{1}{2}}l_{j}$, we have
$$
B \Big[\sqrt{\frac{k !}{\alpha !}} P_{+}\left(h_{1}^{\alpha_{1}} \otimes h_{2}^{\alpha_{2}} \otimes \cdots\otimes h_{n}^{\alpha_{n}}\right)\Big]
=\sqrt{\frac{1}{\alpha!}}[Bh_1]^{\alpha_1}\cdots[Bh_n]^{\alpha_n}.
$$
Thus, for any $f_1,\cdots,f_n \in \mathcal{H}$ we have
\begin{equation}\label{equationmap}
B\Big[\sqrt{\frac{k !}{\alpha !}} P_{+}\left(f_{1}^{\alpha_{1}} \otimes \cdots\otimes f_{n}^{\alpha_{n}}\right)\Big]
=\sqrt{\frac{1}{\alpha!}}[Bf_1]^{\alpha_1}\cdots[Bf_n]^{\alpha_n}.
\end{equation}
For any monomial
$$\sqrt{\frac{1}{2^k\alpha!}}[l_{1}(z)]^{\alpha_1}\cdots[l_{n}(z)]^{\alpha_n},$$
we have
\begin{align*}
&Be^{it d\Gamma(H)}B^{-1}\sqrt{\frac{1}{2^k\alpha!}}[l_{1}]^{\alpha_1}\cdots[l_{n}]^{\alpha_n}\\
=&Be^{it d\Gamma(H)}B^{-1}B \Big[\sqrt{\frac{k !}{\alpha !}} P_{+}\left(h_{1}^{\alpha_{1}} \otimes \cdots\otimes h_{n}^{\alpha_{n}}\right)\Big]\\
=&B \Big[\sqrt{\frac{k !}{\alpha !}} P_{+}\left([e^{itH}h_{1}]^{\alpha_{1}} \otimes \cdots\otimes [e^{it H}h_{n}]^{\alpha_{n}}\right)\Big]\\
=&\sqrt{\frac{1}{\alpha!}}[Be^{itH}h_1]^{\alpha_1}\cdots[Be^{itH}h_n]^{\alpha_n}\quad (\text{by (\ref{equationmap})}) \\
=&\sqrt{\frac{1}{\alpha!}}[Be^{itH}B^{-1}Bh_1]^{\alpha_1}\cdots[Be^{itH}B^{-1}Ih_n]^{\alpha_n}\\
=&\sqrt{\frac{1}{2^k \alpha!}}[Be^{itH}B^{-1}l_1]^{\alpha_1}\cdots[Be^{itH}B^{-1}l_n]^{\alpha_n}.
\end{align*}
Then, for any polynomial $p$, we have
$$
Be^{it d\Gamma(H)}B^{-1}p(z)=p(Be^{itH}B^{-1}l_1(z),\cdots ,Be^{itH}B^{-1}l_n(z),\cdots ).
$$
For any positive integer $n$ and $f\in F^2(\mathbb{C}^{n},d\lambda_{n})$,
there is a sequence of polynomials $\{p_n\}$ such that
$$\lim_{n}p_n(z)=f(z)\text{\quad for any $z\in \mathbb{C}^{n}$ and\quad}\lim_{n}p_n=f\text{\quad in $F^2(\mathbb{C}^{n},d\lambda_{n})$}. $$
Thus we have
$$
Be^{it d\Gamma(H)}B^{-1}f(z)=f(Be^{itH}B^{-1}l_1(z),\cdots ,Be^{itH}B^{-1}l_n(z),\cdots ).
$$
\end{proof}
Let $\mu$ be a real number,
the Gibbs grand canonical equilibrium state is defined in terms of the generalized Hamiltonian $K_{\mu}=d \Gamma(H-\mu I)$ whenever
$$
\exp \left\{-\beta d \Gamma(H-\mu I)\right\}
$$
is trace-class. This latter property places a constraint on the possible values of $\mu$.
\begin{proposition}[{\cite[Proposition 5.2.27]{Bartteli}}]\label{traceclass}
Let H be a self-adjoint operator on the Hilbert space $\mathcal{H}$
and let $\beta \in \mathbb{R} .$ The following conditions are equivalent:\\
(1) $e^{-\beta H}$ is of trace-class on $\mathcal{H}$ and $\beta(H-\mu I)>0$,\\
(2) $e^{-\beta d \Gamma(H-\mu I)}$ is of trace-class on $\mathcal{F}_{+}(\mathcal{H})$ for all $\mu \in \mathbb{R}$.
\end{proposition}
Let us now assume that $\exp \left\{-\beta K_{\mu}\right\}$ is of trace-class and then calculate the Gibbs state
$$
\omega(A)=\frac{\operatorname{Tr}\left(e^{-\beta d \Gamma(H-\mu I)} A\right)}{\operatorname{Tr}\left(e^{-\beta d \Gamma(H-\mu I)}\right)}
$$
for any operator $A$ on $\mathcal{F}_{+}(\mathcal{H})$ such that $e^{-\beta d \Gamma(H-\mu I)} A$ is of trace-class.
By Proposition \ref{traceclass}, we have $e^{-\beta(H-\mu I)}$ is a bounded operator.
The Gibbs equilibrium state is important in the quantum physics, see \cite{Bartteli}.
As an application of the Fock space on $\mathbb{C}^{\infty}$, we are going to give a trace formula and apply it to the Gibbs equilibrium state.

For any $z\in \mathbb{C}^{\infty}$ and $x\in \mathbb{C}^{\infty}(d\lambda_{\infty})$, let
$$K_{x}(z)=e^{\frac{1}{2}\langle x,z\rangle}.$$
Let $x^{(n)}=(x_1,\cdots,x_n,0,0,\cdots)$ and $z^{(n)}=(z_1,\cdots,z_n,0,0,\cdots)$,
then $K_{x^{(n)}}(z^{(n)})$ is the reproducing kernel of $F^2(\mathbb{C}^n,d\lambda_{n}).$ Let $X^{(n)}$ be an operator on $F^2(\mathbb{C}^n,d\lambda_{n})$, by \cite[Proposition 4]{zhu}, we have
$$\operatorname{Tr}(X^{(n)})=\int_{\mathbb{C}^n}\langle X^{(n)}K_{x^{(n)}},K_{x^{(n)}} \rangle d\lambda_{n}(x^{(n)}).$$

\begin{theorem}\label{trace}
If $X$ is an operator on $F^2(\mathbb{C}^{\infty},d\lambda_{\infty})$, then $X$ is of trace-class if and only if
$$|\lim_{n\rightarrow\infty}\int_{\mathbb{C}^n}\langle X K_{x^{(n)}},K_{x^{(n)}} \rangle d\lambda_{n}(x^{(n)})|<\infty.$$
In that case, we have
$$\operatorname{Tr}(X)=\lim_{n\rightarrow\infty}\int_{\mathbb{C}^n}\langle X K_{x^{(n)}},K_{x^{(n)}} \rangle d\lambda_{n}(x^{(n)}).$$
Moreover,  if $A$ is an operator on the Bose-Fock space, then we have
\begin{align*}
&\operatorname{Tr}(e^{-\beta d \Gamma(H-\mu I)} A)\\
=&\lim_{n\rightarrow\infty}\int_{\mathbb{C}^n}\int_{\mathbb{C}^{\infty}} (BAB^{-1} K_{x^{(n)}})(z)\overline{K_{x^{(n)}}\Big(Be^{-\beta (H-\mu I)}B^{-1}l_1(z),\cdots\Big)} d\lambda_{\infty}(z) d\lambda_{n}(x^{(n)}).
\end{align*}
\end{theorem}
\begin{proof}Let
$$
E_n=\left\{e_{\alpha_1}\circ l_1\times e_{\alpha_2}\circ l_2\times\cdots\times e_{\alpha_n}\circ l_n: \quad  \alpha_{1},\cdots,\alpha_{n}\geq 1\right\} \text{ when $n\geq 1$ and } E_0=\{\mathbb{I}\},$$
where $\mathbb{I}$ stands for the function which maps all $z\in\mathbb{C}^{\infty}$ to $1$.
It is easy to see that $E_n$ is the set of basis of $F^2(\mathbb{C}^n,d\lambda_n)\ominus F^2(\mathbb{C}^{n-1},d\lambda_n).$ Thus
we know that $\bigcup_{n=0}^{\infty}E_n$ is a basis set of the Fock space on $\mathbb{C}^{\infty}$. If $X$ is of trace-class, then
$$
\operatorname{Tr}(X)=\sum_{n=0}^{\infty} \sum_{e\in E_n}\langle X e,e \rangle
=\lim_{m\rightarrow\infty}\sum_{n=0}^{m} \sum_{e\in E_n}\langle X e,e \rangle.
$$
Since $P_{m}$ is the projection from $F^2(\mathbb{C}^{\infty}, d\lambda_{\infty})$ to $F^2(\mathbb{C}^{m}, d\lambda_{m})$,
we have $P_{m}XP_{m}$ is an operator from $F^2(\mathbb{C}^{m}, d\lambda_{m})$ to $F^2(\mathbb{C}^{m}, d\lambda_{m})$. Since $\bigcup_{n=0}^m E_n$ is a basis set of $F^2(\mathbb{C}^m,d\lambda_m)$, we have
\begin{align*}
\sum_{n=0}^{m} \sum_{e\in E_n}\langle X e,e \rangle&=\sum_{n=0}^{m} \sum_{e\in E_n}\langle P_mXP_m e,e \rangle=\operatorname{Tr}(P_mXP_m)\\
&=\int_{\mathbb{C}^m}\langle P_mXP_m K_{x^{(m)}},K_{x^{(m)}} \rangle d\lambda_{m}(x^{(m)})\\
&=\int_{\mathbb{C}^m}\langle X K_{x^{(m)}},K_{x^{(m)}} \rangle d\lambda_{m}(x^{(m)}).
\end{align*}
Thus
$$|\lim_{m\rightarrow\infty}\int_{\mathbb{C}^m}\langle X K_{x^{(m)}},K_{x^{(m)}} \rangle d\lambda_{m}(x^{(m)})|=|\lim_{m\rightarrow\infty}\sum_{n=0}^{m} \sum_{e\in E_n}\langle X e,e \rangle|=|\operatorname{Tr}(X)| <\infty.$$
On the other hand, if
$$|\lim_{n\rightarrow\infty}\int_{\mathbb{C}^n}\langle X K_{x^{(n)}},K_{x^{(n)}} \rangle|<\infty,$$
then, by the argument above, we know that $X$ is of trace-class.

The second conclusion follows from Theorem \ref{second quantization}.
\end{proof}

If $e^{-\beta d \Gamma(H-\mu I)}$ is of trace-class,
by Proposition \ref{traceclass}, we know that $\beta (H-\mu I)$ is a positive operator and $e^{-\beta H}$ is of trace-class on $\mathcal{H}$.
We suppose that $\{\lambda_j:j=1,\cdots,n,\cdots\}$ and $\{v_j:j=1,\cdots,n,\cdots\}$ are eigenvalues and eigenvectors of
$e^{-\beta (H-\mu I)}$ such that
\begin{equation}\label{eigenvector}
e^{-\beta (H-\mu I) }v_j =\lambda_jv_j.
\end{equation}
Thus $0<\lambda_j<1$.
Since $e^{-\beta (H-\mu I)}$ is a self-adjoint operator, we suppose $\{v_j\}$ is a basis of $\mathcal{H}$.

Next, we are going to study the Gibbs state.
Let $B_{H}$ denote the isomorphism from the Bose-Fock space to the Fock space over $\mathbb{C}^{\infty}$ such that
$$
B_{H} \Big[\sqrt{\frac{k !}{\alpha !}} P_{+}\left(v_{1}^{\alpha_{1}} \otimes\cdots \otimes v_{n}^{\alpha_{n}}\right)\Big](z)
=e_{\alpha_1}\circ l_1\cdots e_{\alpha_n}\circ l_n,
$$
where $k=\sum_{m}\alpha_{m}$. Because $B^{-1}_{H} l_j=\sqrt{2}B^{-1}_{H} e_{1}\circ l_j=\sqrt{2}v_j$, we have
$$B_{H}e^{-\beta (H-\mu I)} B^{-1}_{H} l_j =\lambda_j l_j. $$
By the construction above and Theorem \ref{trace}, we have the following corollary.
\begin{corollary}\label{trace2}
Let $\beta (H-\mu I)$ be a positive operator and $e^{-\beta (H-\mu I)}$ be of trace-class with eigenvalues $\{\lambda_{k}\}$ on $\mathcal{H}$.
For any operator $A$, we have
\begin{align*}
&\operatorname{Tr}(e^{-\beta d \Gamma(H-\mu I)} A)\\
=&\lim_{n\rightarrow\infty}\int_{\mathbb{C}^n}\int_{\mathbb{C}^{\infty}} (B_{H}AB_{H}^{-1} K_{x^{(n)}})(z)\overline{K_{x^{(n)}}(\lambda_1z_1,\cdots,\lambda_nz_n,0,0,\cdots)} d\lambda_{\infty}(z) d\lambda_{n}(x^{(n)}).
\end{align*}
\end{corollary}

In \cite[Proposition 5.2.28]{Bartteli}, if $e^{-\beta d \Gamma(H-\mu I)}$ is of trace-class on the Bose-Fock space, we have the following formulas
$$\omega(W(f))=\exp\{-\Big\langle f,(1+ e^{-\beta (H-\mu I)})(1- e^{-\beta( H-\mu I)})^{-1} f\Big\rangle / 4\},$$
and
$$
\omega(a_{+}^*(f)a_{+}(g))=\Big\langle g, e^{-\beta (H-\mu I)}(1- e^{-\beta( H-\mu I)})^{-1} f\Big\rangle.
$$
Next, we are going to show that Theorem \ref{trace} implies these two formulas. Moreover, we will generalize the second one.
\begin{corollary}\label{Weylgibbs}
If $e^{-\beta d \Gamma(H-\mu I)}$ is of trace-class on the Bose-Fock space, then the Gibbs state of a Weyl operator $W(f)$ is given by
$$\omega(W(f))=\exp\{-\Big\langle f,(1+ e^{-\beta (H-\mu I)})(1- e^{-\beta( H-\mu I)})^{-1} f\Big\rangle / 4\}.$$
\end{corollary}
\begin{proof}
Let $\{v_{j}\}$ be a basis as in (\ref{eigenvector}), we suppose $f=\sum_{j}f_{j}v_{j}$.
Let $y=(-i\overline{f_1},\cdots,-i\overline{f_n},\cdots)\in \mathbb{C}^{\infty}$, by Theorem \ref{weyl} we have
$$B_{H}W(h)B^{-1}_{H}=U_{y}.$$
Then, by Corollary \ref{trace2}, we have
\begin{align*}
&\operatorname{Tr}(e^{-\beta d \Gamma(H-\mu I)} W(f))\\
=&\lim_{n\rightarrow\infty}\int_{\mathbb{C}^n}\int_{\mathbb{C}^{\infty}} (U_y K_{x^{(n)}})(z)\overline{K_{x^{(n)}}(\lambda_1z_1,\cdots,\lambda_nz_n,0,0,\cdots)} d\lambda_{\infty}(z) d\lambda_{n}(x^{(n)}).
\end{align*}
By the definition of $U_{y}$, we have
\begin{align*}
&\int_{\mathbb{C}^{\infty}} (U_y K_{x^{(n)}})(z)\overline{K_{x^{(n)}}(\lambda_1z_1,\cdots,\lambda_nz_n,0,0,\cdots)} d\lambda_{\infty}(z)\\
=&\int_{\mathbb{C}^{\infty}} e^{\frac{1}{2}\sum_{j=1}^n \overline{x_{j}}(z_j-y_j) +\frac{1}{2}\sum_{j=1}^{\infty} \overline{y_{j}}l_j(z)-\frac{1}{4}|y|^2 } e^{\frac{1}{2}\sum_{j=1}^{n} x_{j}\lambda_{j}\overline{z_{j}}} d\lambda_{\infty}(z).
\end{align*}
Since
$$\int_{\mathbb{C}} e^{\frac{1}{2}\overline{y_{j}}l_j(z)} d\lambda_{1}(z_j)=1 \text{ for any }j,$$
we have
\begin{align*}
\int_{\mathbb{C}^{\infty}} e^{\frac{1}{2}\sum_{j=n+1}^{\infty} \overline{y_{j}}l_j(z)} d\lambda_{\infty}(z)
=\int_{\mathbb{C}^{\infty}} e^{\frac{1}{2}\sum_{j=m}^{\infty} \overline{y_{j}}l_j(z)} d\lambda_{\infty}(z)
\end{align*}
for any integer $m$ and
\begin{align*}
\lim_{m\rightarrow\infty}|\int_{\mathbb{C}^{\infty}} e^{\frac{1}{2}\sum_{j=m}^{\infty} \overline{y_{j}}l_j(z)} d\lambda_{\infty}(z)-1|
=&\lim_{m\rightarrow\infty}\int_{\mathbb{C}^{\infty}} |e^{\frac{1}{2}\sum_{j=m}^{\infty} \overline{y_{j}}l_j(z)}-1| d\lambda_{\infty}(z)\\
\leq&\lim_{m\rightarrow\infty}\|e^{\frac{1}{2}\sum_{j=m}^{\infty} \overline{y_{j}}l_j}-1  \|_{F^2(\mathbb{C}^{\infty},d\lambda_{\infty})}\\
\leq&\lim_{m\rightarrow\infty}\|e^{\frac{1}{4}\sum_{j=m}^{\infty}|y_{j}|^2}U_{y-y^{(m)}}1-1  \|_{F^2(\mathbb{C}^{\infty},d\lambda_{\infty})}\\
=&0\text{ (by Lemma \ref{weyloperator})}.
\end{align*}
Thus
$$
\int_{\mathbb{C}^{\infty}} e^{\frac{1}{2}\sum_{j=n+1}^{\infty} \overline{y_{j}}l_j(z)} d\lambda_{\infty}(z)
=\lim_{m\rightarrow\infty}\int_{\mathbb{C}^{\infty}} e^{\frac{1}{2}\sum_{j=m}^{\infty} \overline{y_{j}}l_j(z)} d\lambda_{\infty}(z)=1.
$$
Then, we have
\begin{align*}
&\int_{\mathbb{C}^{\infty}} (U_y K_{x^{(n)}})(z)\overline{K_{x^{(n)}}(\lambda_1z_1,\cdots,\lambda_nz_n,0,0,\cdots)} d\lambda_{\infty}(z)\\
=&\int_{\mathbb{C}^{n}} e^{\frac{1}{2}\sum_{j=1}^n \overline{x_{j}}(z_j-y_j) +\frac{1}{2}\sum_{j=1}^{n}
\overline{y_{j}}z_j-\frac{1}{4}|y|^2 } e^{\frac{1}{2}\sum_{j=1}^{n} x_{j}\lambda_{j}\overline{z_{j}}} d\lambda_{n}(z)\\
=& e^{\frac{1}{2}\sum_{j=1}^n \overline{x_{j}}(\lambda_{j}x_j-y_j) +\frac{1}{2}\sum_{j=1}^{n}
\overline{y_{j}}\lambda_{j}x_j-\frac{1}{4}|y|^2}.
\end{align*}
Thus
\begin{align*}
&\operatorname{Tr}(e^{-\beta d \Gamma(H-\mu I)} W(f))\\
=&\lim_{n\rightarrow\infty}\int_{\mathbb{C}^n}e^{\frac{1}{2}\sum_{j=1}^n \overline{x_{j}}(\lambda_{j}x_j-y_j) +\frac{1}{2}\sum_{j=1}^{n}
\overline{y_{j}}\lambda_{j}x_j-\frac{1}{4}|y|^2 } d\lambda_{n}(x^{(n)})\\
=&\lim_{n\rightarrow\infty}\int_{\mathbb{C}^n}e^{\frac{1}{2}\sum_{j=1}^n \overline{x_{j}}(\lambda_{j}x_j-y_j) +\frac{1}{2}\sum_{j=1}^{n}
\overline{y_{j}}\lambda_{j}x_j-\frac{1}{4}|y|^2 } \frac{1}{(2\pi)^{2n}}e^{-\frac{1}{2}\sum_{j=1}^{n}|x_{j}|^2}d(x^{(n)})\\
=&\lim_{n\rightarrow\infty}\int_{\mathbb{C}^n}e^{\frac{1}{2}\sum_{j=1}^n \overline{x_{j}}(-y_j) +\frac{1}{2}\sum_{j=1}^{n}
\overline{y_{j}}\lambda_{j}x_j-\frac{1}{4}|y|^2 } \frac{1}{(2\pi)^{2n}}e^{-\frac{1}{2}\sum_{j=1}^{n}(1-\lambda_j)|x_{j}|^2}d(x^{(n)})\\
=&\lim_{n\rightarrow\infty}\frac{1}{\prod_{j=1}^{n}\sqrt{1-\lambda_j}} \int_{\mathbb{C}^n}e^{\frac{1}{2}\sum_{j=1}^n \overline{\frac{x_{j}}{\sqrt{1-\lambda_j}}}(-y_j) +\frac{1}{2}\sum_{j=1}^{n}
\overline{y_{j}}\lambda_{j}\frac{x_j}{\sqrt{1-\lambda_j}}-\frac{1}{4}|y|^2 } d\lambda_{n}(x^{(n)})\\
=&\lim_{n\rightarrow\infty}\frac{1}{\prod_{j=1}^{n}\sqrt{1-\lambda_j}}
e^{-\frac{1}{4}|y|^2-\sum_{j=1}^{n}\frac{1}{2}\frac{\lambda_j}{1-\lambda_j}|y_j|^2 }\\
=&\frac{1}{\prod_{j=1}^{\infty}\sqrt{1-\lambda_j}}
e^{-\frac{1}{4}|y|^2-\sum_{j=1}^{\infty}\frac{1}{2}\frac{\lambda_j}{1-\lambda_j}|y_j|^2 }\\
=&\frac{1}{\prod_{j=1}^{\infty}\sqrt{1-\lambda_j}}
e^{-\frac{1}{4}\sum_{j=1}^{\infty}\frac{1+\lambda_j}{1-\lambda_j}|y_j|^2 },
\end{align*}
where $\prod_{j=1}^{\infty}\sqrt{1-\lambda_j}$ is convergent because $\sum_{j}\lambda_{j}=\operatorname{Tr}(e^{-\beta(H-\mu I)})<\infty $ and $\lambda_j<1$. By functional calculus, we have
\begin{align*}
\Big\langle f,(1+ e^{-\beta (H-\mu I)})(1- e^{-\beta (H-\mu I)})^{-1} f\Big\rangle
=&\Big\langle f,(1+ e^{-\beta (H-\mu I)})(1- e^{-\beta (H-\mu I)})^{-1} \sum_{j}f_jv_j\Big\rangle\\
=& \sum_{j}f_j \Big\langle f,(1+ e^{-\beta (H-\mu I)})(1- e^{-\beta (H-\mu I)})^{-1}v_j\Big\rangle\\
=& \sum_{j}f_j \Big\langle f,(1+ \lambda_j )(1- \lambda_j)^{-1}v_j\Big\rangle\\
=& \sum_{j}|f_j|^2\frac{1+\lambda_j}{1-\lambda_j}= \sum_{j}|y_j|^2\frac{1+\lambda_j}{1-\lambda_j}.\\
\end{align*}
Thus, we have
$$
\operatorname{Tr}(e^{-\beta d \Gamma(H-\mu I)} W(f))=\frac{1}{\prod_{j=1}^{\infty}\sqrt{1-\lambda_j}}
e^{-\frac{1}{4}\Big\langle f,(1+ e^{-\beta (H-\mu I)})(1- e^{-\beta( H-\mu I)})^{-1} f\Big\rangle}.
$$
Since $W(f)=I$ when $f=0$, we have
$$
\operatorname{Tr}(e^{-\beta d \Gamma(H-\mu I)})=\frac{1}{\prod_{j=1}^{\infty}\sqrt{1-\lambda_j}}.
$$
Then we have
$$
\omega(W(f))=\frac{\operatorname{Tr}(e^{-\beta d \Gamma(H-\mu I)} W(f))}{\operatorname{Tr}(e^{-\beta d \Gamma(H-\mu I)})}
=e^{-\frac{1}{4}\Big\langle f,(1+ e^{-\beta (H-\mu I)})(1- e^{-\beta (H-\mu I)})^{-1} f\Big\rangle}.
$$
\end{proof}

\begin{corollary}
If $e^{-\beta d \Gamma(H-\mu I)}$ is of trace-class on the Bose-Fock space, then
\begin{align*}
&\omega\Big(a_{+}^*(f^{(1)})\cdots a_{+}^*(f^{(m)}) a_{+}(g^{(1)})\cdots a_{+}(g^{(m)})\Big)\\
=&\Big\langle\sqrt{m !}P_{+}\left(\tilde{f}^{(1)} \otimes \cdots\otimes \tilde{f}^{(m)}\right),
\sqrt{m !}P_{+}\left(\tilde{g}^{(1)} \otimes \cdots\otimes \tilde{g}^{(m)}\right)\Big\rangle_{\mathcal{F}_{+}(\mathcal{H})},
\end{align*}
where
$$\tilde{f}^{(k)}=\frac{e^{-\beta (H-\mu I)}}{\sqrt{1-e^{-\beta (H-\mu I)}}} f^{(k)}
\text{ and }
\tilde{g}^{(k)}=\frac{1}{\sqrt{1-e^{-\beta (H-\mu I)}}} g^{(k)}.
$$
\end{corollary}
\begin{proof}
Let $\widetilde{x}^{(n)}=(\lambda_1x_1,\cdots,\lambda_nx_n,0,\cdots)$,
$$T_1=a_{+}(f^{(1)})\cdots a_{+}(f^{(m)})\text{ and } T_2=a_{+}(g^{(1)})\cdots a_{+}(g^{(m)}).$$
Since
$$K_{x^{(n)}}(\lambda_1z_1,\cdots,\lambda_nz_n,0,0,\cdots)=e^{\frac{1}{2}\sum_{j=1}^{n}\overline{x_j}\lambda_jz_j}
=K_{\widetilde{x}^{(n)}}(z),$$
we have
\begin{align*}
&\operatorname{Tr}(e^{-\beta d \Gamma(H-\mu I)} T_1^*T_2)\text{ (by Corollary \ref{trace2})}\\
=&\lim_{n\rightarrow\infty}\int_{\mathbb{C}^n}\int_{\mathbb{C}^{\infty}} (B_{H}T^*_1T_2B_{H}^{-1} K_{x^{(n)}})(z)\overline{K_{x^{(n)}}(\lambda_1z_1,\cdots,\lambda_nz_n,0,0,\cdots)} d\lambda_{\infty}(z) d\lambda_{n}(x^{(n)})\\
=&\lim_{n\rightarrow\infty}\int_{\mathbb{C}^n}\int_{\mathbb{C}^{\infty}} (B_{H}T^*_1T_2B_{H}^{-1} K_{x^{(n)}})(z)\overline{K_{\widetilde{x}^{(n)}}(z)} d\lambda_{\infty}(z) d\lambda_{n}(x^{(n)})\\
=&\lim_{n\rightarrow\infty}\int_{\mathbb{C}^n}\int_{\mathbb{C}^{\infty}} (B_H T_{2}B^{-1}_H K_{x^{(n)}})(z)\overline{(B_H T_{1}B^{-1}_HK_{\widetilde{x}^{(n)}})(z)} d\lambda_{\infty}(z) d\lambda_{n}(x^{(n)}).\\
\end{align*}
By Proposition \ref{creation}, we have
$$B_HT_1B^{-1}_H=T_{\overline{B_Hf^{(1)}}}\cdots T_{\overline{B_Hf^{(m)}}}\text{ and }
 B_HT_2B^{-1}_H=T_{\overline{B_Hg^{(1)}}}\cdots T_{\overline{B_Hg^{(m)}}}.$$
We suppose $g^{(k)}=\sum_{j}g^{(k)}_{j}v_{j}$ and $f^{(k)}=\sum_{j}f^{(k)}_{j}v_{j}$.
For any polynomial $p$, we have
\begin{align*}
&\langle p, T_{\overline{B_H g^{(m)}}} K_{x^{(n)}}\rangle_{F^{2}(\mathbb{C}^{\infty},d\lambda_{\infty})}
=\langle T_{B_H g^{(m)}}p ,K_{x^{(n)}}\rangle_{F^{2}(\mathbb{C}^{\infty},d\lambda_{\infty})}\\
&=\sum_{j=1}^{n}g^{(m)}_j\frac{x_j}{\sqrt{2}}p(x^{(n)})=\sum_{j=1}^{n}g^{(m)}_j\frac{x_j}{\sqrt{2}} \langle p ,K_{x^{(n)}}\rangle_{F^{2}(\mathbb{C}^{\infty},d\lambda_{\infty})}\\
&=\langle p ,\sum_{j=1}^{n}\overline{g^{(m)}_j\frac{x_j}{\sqrt{2}}}K_{x^{(n)}}\rangle_{F^{2}(\mathbb{C}^{\infty},d\lambda_{\infty})},
\end{align*}
which means $(T_{\overline{B_H g^{(m)}}} K_{x^{(n)}})(z)=\sum_{j=1}^{n}\overline{g^{(m)}_j\frac{x_j}{\sqrt{2}}}K_{x^{(n)}}(z)$.
Thus, we have
$$(B_H T_{2}B^{-1}_H K_{x^{(n)}})(z)=\Big(\sum_{j=1}^{n}\overline{g^{(1)}_j\frac{x_j}{\sqrt{2}}}\Big)
\cdots\Big(\sum_{j=1}^{n}\overline{g^{(m)}_j\frac{x_j}{\sqrt{2}}}\Big)K_{x^{(n)}}(z)$$
and
$$(B_H T_{1}B^{-1}_H K_{\widetilde{x}^{(n)}})(z)=\Big(\sum_{j=1}^{n}\overline{f^{(1)}_j\frac{\lambda_jx_j}{\sqrt{2}}}\Big)
\cdots\Big(\sum_{j=1}^{n}\overline{f^{(m)}_j\frac{\lambda_jx_j}{\sqrt{2}}}\Big)K_{\widetilde{x}^{(n)}}(z).$$
Then
\begin{align*}
&\operatorname{Tr}\Big(e^{-\beta d \Gamma(H-\mu I)} a_{+}^*(f^{(1)})\cdots a_{+}^*(f^{(m)}) a_{+}(g^{(1)})\cdots a_{+}(g^{(m)})\Big)\\
=&\lim_{n\rightarrow\infty}\int_{\mathbb{C}^n}\int_{\mathbb{C}^{\infty}}
\prod_{k=1}^m\Big(\sum_{j=1}^{n}\overline{g^{(k)}_j\frac{x_j}{\sqrt{2}}}\Big)K_{x^{(n)}}(z)
\overline{\prod_{k=1}^m\Big(\sum_{j=1}^{n}\overline{f^{(k)}_j\frac{\lambda_jx_j}{\sqrt{2}}}\Big)K_{\widetilde{x}^{(n)}}(z)}
d\lambda_{\infty}(z) d\lambda_{n}(x^{(n)})\\
=&\lim_{n\rightarrow\infty}\int_{\mathbb{C}^n}\prod_{k=1}^m\Big(\sum_{j=1}^{n}f^{(k)}_j\frac{\lambda_jx_j}{\sqrt{2}}\Big)
\overline{\prod_{k=1}^m\Big(\sum_{j=1}^{n}g^{(k)}_j\frac{x_j}{\sqrt{2}}\Big)}\int_{\mathbb{C}^{\infty}} K_{x^{(n)}}(z)\overline{
K_{\widetilde{x}^{(n)}}(z)} d\lambda_{\infty}(z) d\lambda_{n}(x^{(n)})\\
=&\lim_{n\rightarrow\infty}\int_{\mathbb{C}^n}\prod_{k=1}^m\Big(\sum_{j=1}^{n}f^{(k)}_j\frac{\lambda_jx_j}{\sqrt{2}}\Big)
\overline{\prod_{k=1}^m\Big(\sum_{j=1}^{n}g^{(k)}_j\frac{x_j}{\sqrt{2}}\Big)}
e^{\frac{1}{2}\sum_{j=1}^{n}\lambda_j|x_j|^2} d\lambda_{n}(x^{(n)})\\
=&\frac{1}{\prod_{j=1}^{\infty}\sqrt{1-\lambda_j}}
\lim_{n\rightarrow\infty}\int_{\mathbb{C}^n}\prod_{k=1}^m\Big(\sum_{j=1}^{n}\frac{\lambda_j f^{(k)}_j}{\sqrt{1-\lambda_j}}\frac{x_j}{\sqrt{2}}\Big)
\overline{\prod_{k=1}^m\Big(\sum_{j=1}^{n}\frac{g^{(k)}_j}{\sqrt{1-\lambda_j}}\frac{x_j}{\sqrt{2}}\Big)} d\lambda_{n}(x^{(n)}).\\
\end{align*}
By the proof of Corollary \ref{Weylgibbs}, we have
$$
\operatorname{Tr}(e^{-\beta d \Gamma(H-\mu I)})=\frac{1}{\prod_{j=1}^{\infty}\sqrt{1-\lambda_j}},
$$
thus
\begin{align*}
&\omega\Big(a_{+}^*(f^{(1)})\cdots a_{+}^*(f^{(m)}) a_{+}(g^{(1)})\cdots a_{+}(g^{(m)})\Big)\\
=&\lim_{n\rightarrow\infty}\int_{\mathbb{C}^n}\prod_{k=1}^m\Big(\sum_{j=1}^{n}\frac{\lambda_j f^{(k)}_j}{\sqrt{1-\lambda_j}}\frac{x_j}{\sqrt{2}}\Big)
\overline{\prod_{k=1}^m\Big(\sum_{j=1}^{n}\frac{g^{(k)}_j}{\sqrt{1-\lambda_j}}\frac{x_j}{\sqrt{2}}\Big)} d\lambda_{n}(x^{(n)})\\
=&\lim_{n\rightarrow\infty}\int_{\mathbb{C}^{\infty}}\prod_{k=1}^m
\Big(\sum_{j=1}^{n}\frac{\lambda_j f^{(k)}_j}{\sqrt{1-\lambda_j}}e_1\circ l_j(x)\Big)
\overline{\prod_{k=1}^m\Big(\sum_{j=1}^{n}\frac{g^{(k)}_j}{\sqrt{1-\lambda_j}}e_1\circ l_j(x)\Big)} d\lambda_{\infty}(x)\\
=&\lim_{n\rightarrow\infty} \Big\langle\prod_{k=1}^m\sum_{j=1}^{n}\frac{\lambda_j f^{(k)}_j}{\sqrt{1-\lambda_j}}e_1\circ l_j,
 \prod_{k=1}^m\sum_{j=1}^{n}\frac{g^{(k)}_j}{\sqrt{1-\lambda_j}}e_1\circ l_j\Big\rangle_{F^2(\mathbb{C}^{\infty},d\lambda_{\infty})}.
\end{align*}
Let
$$\tilde{f}^{(k)}=\frac{e^{-\beta (H-\mu I)}}{\sqrt{1-e^{-\beta (H-\mu I)}}} f^{(k)}
\text{ and }
\tilde{g}^{(k)}=\frac{1}{\sqrt{1-e^{-(H-\mu I)}}} g^{(k)}.
$$
By functional calculus, we have
$$B_H\widetilde{f}^{(k)}=B_H \frac{e^{-\beta (H-\mu I)}}{\sqrt{1-e^{-\beta (H-\mu I)}}} f^{(k)}
=B_H\sum_{j=1}^{\infty}\frac{\lambda_j f^{(k)}_j}{\sqrt{1-\lambda_j}}v_j
=\sum_{j=1}^{\infty}\frac{\lambda_j f^{(k)}_j}{\sqrt{1-\lambda_j}}e_1\circ l_j $$
and
$$B_H\widetilde{g}^{(k)}=B_H \frac{1}{\sqrt{1-e^{-\beta (H-\mu I)}}} g^{(k)}
=B_H\sum_{j=1}^{\infty}\frac{ g^{(k)}_j}{\sqrt{1-\lambda_j}}v_j
=\sum_{j=1}^{\infty}\frac{ g^{(k)}_j}{\sqrt{1-\lambda_j}}e_1\circ l_j.$$
Let $Q_n$ be a projection on $\chi_1$ such that
$$Q_n\sum_{j=1}^{\infty}a_j e_1\circ l_j=\sum_{j=1}^{n}a_j e_1\circ l_j  $$
for any $\sum_{j}a_j e_1\circ l_j\in \chi_1$, further we denote $\hat{Q}_n=B^{-1}_HQ_nB_H$.
Then, we have
$$\prod_{k=1}^m\sum_{j=1}^{n}\frac{\lambda_j f^{(k)}_j}{\sqrt{1-\lambda_j}}e_1\circ l_j
=\prod_{k=1}^m Q_n\sum_{j=1}^{\infty}\frac{\lambda_j f^{(k)}_j}{\sqrt{1-\lambda_j}}e_1\circ l_j=\prod_{k=1}^mQ_nB_H\widetilde{f}^{(k)}
=\prod_{k=1}^mB_H \hat{Q}_n\widetilde{f}^{(k)}$$
and
$$\prod_{k=1}^m\sum_{j=1}^{n}\frac{ g^{(k)}_j}{\sqrt{1-\lambda_j}}e_1\circ l_j
=\prod_{k=1}^m Q_n\sum_{j=1}^{\infty}\frac{ g^{(k)}_j}{\sqrt{1-\lambda_j}}e_1\circ l_j=\prod_{k=1}^mQ_nB_H\widetilde{g}^{(k)}
=\prod_{k=1}^mB_H \hat{Q}_n\widetilde{g}^{(k)}.$$
Thus
\begin{align*}
&\omega\Big(a_{+}^*(f^{(1)})\cdots a_{+}^*(f^{(m)}) a_{+}(g^{(1)})\cdots a_{+}(g^{(m)})\Big)\\
=&\lim_{n\rightarrow\infty} \Big\langle\prod_{k=1}^m\sum_{j=1}^{n}\frac{\lambda_j f^{(k)}_j}{\sqrt{1-\lambda_j}}e_1\circ l_j,
\prod_{k=1}^m\sum_{j=1}^{n}\frac{g^{(k)}_j}{\sqrt{1-\lambda_j}}e_1\circ l_j\Big\rangle_{F^2(\mathbb{C}^{\infty},d\lambda_{\infty})}\\
=&\lim_{n\rightarrow\infty} \Big\langle\prod_{k=1}^mB_H \hat{Q}_n\widetilde{f}^{(k)},
\prod_{k=1}^m B_H \hat{Q}_n\tilde{g}^{(k)}\Big\rangle_{F^2(\mathbb{C}^{\infty},d\lambda_{\infty})}\\
=&\lim_{n\rightarrow\infty}\Big\langle\sqrt{m !}P_{+}\left(\hat{Q}_n\tilde{f}^{(1)} \otimes \cdots\otimes \hat{Q}_n\tilde{f}^{(m)}\right),
\sqrt{m !}P_{+}\left(\hat{Q}_n\tilde{g}^{(1)} \otimes \cdots\otimes \hat{Q}_n\tilde{g}^{(m)}\right)\Big\rangle_{\mathcal{F}_{+}(\mathcal{H})}\\
=&\Big\langle\sqrt{m !}P_{+}\left(\tilde{f}^{(1)} \otimes \cdots\otimes \tilde{f}^{(m)}\right),
\sqrt{m !}P_{+}\left(\tilde{g}^{(1)} \otimes \cdots\otimes \tilde{g}^{(m)}\right)\Big\rangle_{\mathcal{F}_{+}(\mathcal{H})}.\\
\end{align*}
\end{proof}

\section{Fock-Sobolev spaces and Gaussian Harmonic analysis}
In this section, we introduce the Fock-Sobolev spaces and discuss its relationship with Gaussian Harmonic analysis. In \cite{wu}, the authors used the relationship between the Fock-Sobolev spaces and the Gaussian Harmonic analysis to study the boundedness of a integral operator. We will show that, we have a similar conclusion in the infinite dimensional case. As an application of this relationship, we will study the boundedness of creation operators and annihilation operators.

For any $r\in\mathbb{N}$, the Fock Sobolev space on $\mathbb{C}^n$ consists of all $f\in F^{2}(\mathbb{C}^n)$ such that
\begin{equation}\label{Focksobolevnorm}
\|f\|_{F^{2,r}(\mathbb{C}^{n})}= \sum_{k=1}^{r}\Big(\sum_{|\alpha|=k}\int_{\mathbb{C}^{n}}|\partial^{\alpha}f(z)|^2 d\lambda_{n}(z)\Big)^{1/2}<\infty.
\end{equation}
We need to point out that, in some literature the norm of function in the Fock-Sobolev space is defined by
\begin{equation*}
\|f\|_{*}=\sum_{k=1}^{r}\sum_{|\alpha|=k}\Big(\int_{\mathbb{C}^{n}}|\partial^{\alpha}f(z)|^2 d\lambda_{n}(z)\Big)^{1/2}.
\end{equation*}
$\|f\|_{*}$ and $\|f\|_{F^{2,r}(\mathbb{C}^{n})}$ are equivalent for any $n$ but not uniformly equivalent with respect to $n$. Since we need to generalize the Fock-Sobolev space to the infinite dimensional case, we use (\ref{Focksobolevnorm}).

In the Fock space on $\mathbb{C}^n$, we have the following equality
$$ T_{\overline{e_{1}(z_j)}}=\partial_{z_j}.$$
Since the Fock space on $\mathbb{C}^{\infty}$ is not an analytic function space, so we use   the Toeplitz operator $T_{\overline{e_{1}(z_j)}}$ to define the Fock-Sobolev space on $\mathbb{C}^{\infty}$. Let the Fock-Sobolev space $F^{2,r}(\mathbb{C}^{\infty})$ be the completion of finite polynomials with respect to the following norm
$$\|p\|_{F^{2,r}(\mathbb{C}^{\infty})}= \sum_{k=1}^{r}\Big(\sum_{j_1,\cdots,j_k\geq1}\int_{\mathbb{C}^{\infty}}|T_{\overline{e_{1}(z_{j_1})}}\cdots T_{\overline{e_{1}(z_{j_k})}}p(z)|^2 d\lambda_{\infty}(z)\Big)^{1/2}.$$
To study the Fock-Sobolev space, we introduce the Gaussian Sobolev space.
Let Gaussian measure $d\gamma_n$ on $\mathbb{R}^n$ be given by
$$
d\gamma_n(x)=\frac{1}{{(2\pi)}^{\frac{n}{2}}}e^{-\frac{|x|^2}{2}}dx.
$$
The Gaussian Measure can be extended to $\mathbb{R}^\infty$, we denote it by $d\gamma_{\infty}$.
Let $L^2(\mathbb{R}^\infty,d\gamma_{\infty})$ consist of all square-integrable function on $\mathbb{R}^\infty$ with respect to $d\gamma_{\infty}$.
For any $k=0,1, \ldots,$ the Hermite
polynomials $H_{k}$ on the real line are defined by the formula
$$
H_{k}(x)=\frac{(-1)^{k}}{\sqrt{k !}} \exp \left(\frac{x^{2}}{2}\right) \frac{d^{k}}{d x^{k}} \exp \left(-\frac{x^{2}}{2}\right).
$$
By \cite[Lemma 1.3.2 and Corollary 1.3.3]{Bogachev}, we have
\begin{equation}\label{derivative}
H_{k}^{'}(x)=\sqrt{k}H_{k-1}(x)
\end{equation}
and $\left\{H_{k}\right\}$ is an orthonormal basis in the space $L^{2}\left(\mathbb{R},\gamma_{1}\right)$.
By \cite[Lemma 2.5.1 and Example 2.3.5]{Bogachev}, we have
$$\{H_{\alpha_1}(x_1)H_{\alpha_2}(x_2)\cdots H_{\alpha_n}(x_n): \sum \alpha_{j}=k, \quad k=0,1,2,\cdots\}$$
is a basis of $L^2(\mathbb{C}^\infty,d\lambda_{\infty})$.
Let $I_k$ be the projection from $L^2(\mathbb{C}^\infty,d\lambda_{\infty})$ to the subspace which is generated by
$$\{H_{\alpha_1}(x_1)H_{\alpha_2}(x_2)\cdots H_{\alpha_n}(x_n): \sum \alpha_{j}=k\}.$$
By \cite[Chapter 5]{Bogachev}, the Gaussian Sobolev class is the completion of all finite combination of Hermite polynomials with respect to the norm
$$\|f\|_{W^{2,r}(\mathbb{R}^{\infty})}= \sum_{k=1}^{r}\Big(\sum_{j_1,\cdots,j_k\geq1}\int_{\mathbb{C}^{\infty}}|\partial_{x_{j_1}}\cdots \partial_{x_{j_k}}f(x)|^2 d\gamma_{\infty}(x)\Big)^{1/2}.$$
We also have
\begin{equation}\label{Sobolevnorm}
\|f\|_{W^{2,r}(\mathbb{R}^{\infty})}\simeq\|\sum_{k=0}^{\infty} (1+k)^{r/2}I_{k}f\|_{L^2(\mathbb{R}^\infty,d\gamma_{\infty})}.
\end{equation}
The Gaussian Bargmann transform $G$ is an unitary operator  from $L^2(\mathbb{R}^\infty,d\gamma_{\infty})$ to $F^2(\mathbb{C}^\infty,d\lambda_{\infty})$ such that
$$G[H_{\alpha_1}(x_1)H_{\alpha_2}(x_2)\cdots H_{\alpha_n}(x_n)]=e_{\alpha_1}(z_1)\cdots e_{\alpha_n}(z_n).$$
Let $Q_k$ be the projection from $F^2(\mathbb{C}^\infty,d\lambda_{\infty})$ to $\chi_{k}$. We have $GQ_kG^{-1}=I_k$.
It is elementary that $T_{\overline{e_{1}(z_j)}}e_{\alpha_j}(z_j)=\sqrt{\alpha_j} e_{\alpha_{j-1}}(z_j)$, thus by (\ref{derivative}) and the definition of the Fock-Sobolev space and Gaussian Sobolev class we have
\begin{equation}\label{normequality}
\|p\|_{F^{2,r}(\mathbb{C}^{\infty})}=\|G^{-1}p\|_{W^{2,r}(\mathbb{R}^{\infty})}\simeq\|\sum_{k=0}^{\infty} (1+k)^{r/2}Q_{k}p\|_{F^2(\mathbb{C}^\infty,d\lambda_{\infty})}.
\end{equation}
That is to say the Gaussian Bargmann transform $G$ is also an unitary operator  from $W^{2,r}(\mathbb{R}^{\infty})$ to $F^{2,r}(\mathbb{C}^\infty,d\lambda_{\infty})$.

\begin{proposition}For any $\phi\in \chi_{1}$ and $r\geq1$, $T_{\phi}$ and $T_{\overline{\phi}}$ are bounded from $F^{2,r}(\mathbb{C}^{\infty})$ to $F^{2,r-1}(\mathbb{C}^{\infty})$.
\end{proposition}
\begin{proof}Let $\phi=\sum_{j}c_ke_{1}(z_j)$. For any $e_{\alpha_1}(z_1)\cdots e_{\alpha_n}(z_n)\cdots$ with $\sum \alpha_{j}=k$,
we have
\begin{align*}
&\|T_{\phi}e_{\alpha_1}(z_1)\cdots e_{\alpha_n}(z_n) \|_{F^2(\mathbb{C}^{\infty},d\lambda_{\infty})}^2\\
=&\|\sum_{j}c_je_{1}(z_j)e_{\alpha_1}(z_1)\cdots e_{\alpha_n}(z_n)\cdots \|_{F^2(\mathbb{C}^{\infty},d\lambda_{\infty})}^2\\
=&\|\sum_{j}c_j \sqrt{\alpha_{j}+1}e_{\alpha_1}(z_1)\cdots e_{\alpha_{j+1}}(z_j)\cdots \|_{F^2(\mathbb{C}^{\infty},d\lambda_{\infty})}^2\\
=&\sum_{j}|c_j|^2(\alpha_{j}+1)\leq \|\phi\|^2 (k+1)\|e_{\alpha_1}(z_1)\cdots e_{\alpha_n}(z_n)\|_{F^2(\mathbb{C}^{\infty},d\lambda_{\infty})}^2.
\end{align*}
Thus, for any $p_k\in \chi_k$, we have
$$\|T_{\phi}p_k\|_{F^2(\mathbb{C}^{\infty},d\lambda_{\infty})}^2\leq \|\phi\|_{F^2(\mathbb{C}^{\infty},d\lambda_{\infty})}^2 (k+1)\|p_k\|_{F^2(\mathbb{C}^{\infty},d\lambda_{\infty})}^2$$ and $T_{\phi} p_k\in \chi_{k+1}$.
For any polynomial $p$, we have
\begin{align*}
\|T_{\phi}p\|_{F^{2,r-1}(\mathbb{C}^{\infty})}^2
&=\|\sum_{k} T_{\phi}Q_kp\|_{F^{2,r-1}(\mathbb{C}^{\infty})}^2\\
&\backsimeq\|\sum_{k} (k+2)^{\frac{r-1}{2}}T_{\phi}Q_kp\|_{F^{2}(\mathbb{C}^{\infty},d\lambda_{\infty})}^2\text{ (by (\ref{normequality}))}\\
&=\sum_{k}(k+2)^{r-1}\|T_{\phi}Q_kp\|_{F^{2}(\mathbb{C}^{\infty},d\lambda_{\infty})}^2\\
&\leq \sum_{k}(k+2)^{r-1}\|\phi\|_{F^2(\mathbb{C}^{\infty},d\lambda_{\infty})}^2 (k+1)\|Q_kp\|_{F^2(\mathbb{C}^{\infty},d\lambda_{\infty})}^2\\
&\lesssim \|\phi\|_{F^2(\mathbb{C}^{\infty},d\lambda_{\infty})}^2\sum_{k}(k+1)^{r}\|Q_kp\|_{F^2(\mathbb{C}^{\infty},d\lambda_{\infty})}^2\\
&\backsimeq  \|\phi\|_{F^2(\mathbb{C}^{\infty},d\lambda_{\infty})}^2\|p\|^2_{F^{2,r}(\mathbb{C}^{\infty})}\text{ (by (\ref{normequality}))},
\end{align*}
thus $T_{\phi}$ is bounded from $F^{2,r}(\mathbb{C}^{\infty})$ to $F^{2,r-1}(\mathbb{C}^{\infty})$.
The proof for $T_{\overline{\phi}}$ is similar.
\end{proof}

%-----------

\bibliographystyle{amsplain}

\end{document}